\documentclass[10pt]{amsart}

\usepackage{amsmath}
\usepackage{amssymb}
\usepackage{bm}
\usepackage{graphicx}
\usepackage{color}
\definecolor{wine}{rgb}{0.5,0,0.15}

\usepackage{hyperref}
\hypersetup{colorlinks=true, linkcolor=blue, citecolor=magenta, urlcolor=wine}
\usepackage{url}
\usepackage{algpseudocode}
\usepackage{fancyhdr}
\usepackage{mathtools}
\usepackage{tikz-cd}
\usepackage{xy}
\input xy
\xyoption{all}
\usepackage{stmaryrd}
\usepackage{calrsfs}
\usepackage{enumitem}

\newtheorem*{maintheorem*}{Main Theorem}
\newtheorem{theorem}{Theorem}[section]
\newtheorem*{theorem*}{Main Theorem}

\newtheorem{prop}[theorem]{Proposition}

\newtheorem{lemma}[theorem]{Lemma}
\newtheorem{cor}[theorem]{Corollary}

\theoremstyle{definition}
\newtheorem{definition}[theorem]{Definition}

\newtheorem{example}[theorem]{Example}
\numberwithin{equation}{section}

\newcommand{\aaa}{\mathbb{A}}
\newcommand{\cc}{\mathbb{C}}
\newcommand{\ff}{\mathbb{F}}
\newcommand{\nn}{\mathbb{N}}
\newcommand{\pp}{\mathbb{P}}
\newcommand{\qq}{\mathbb{Q}}
\newcommand{\Q}{\mathbb{Q}}
\newcommand{\rr}{\mathbb{R}}
\newcommand{\R}{\mathbb{R}}
\newcommand{\sss}{\mathbb{S}}
\newcommand{\zz}{\mathbb{Z}}
\newcommand{\Z}{\mathbb{Z}}

\newcommand{\Za}{\mathbb{Z}[\alpha]}

\newcommand{\uu}{\mathcal{U}}

\providecommand\ldb{\llbracket}
\providecommand\rdb{\rrbracket}

\newcommand{\gp}{\mathcal{G}}
\newcommand{\qf}{\text{qf}}
\newcommand{\supp}{\text{supp} \, }

\newcommand{\OK}{\mathcal{O}_K} 
\newcommand{\mfp}{\mathfrak{p}}

\DeclareMathOperator{\rank}{rank}


\begin{document}
\raggedbottom
	
\mbox{}
\title{Factorizations in algebraic monogenic semidomains}

\author{Jiya Dani}
\address{Department of Mathematics\\MIT\\Cambridge, MA 02139}
\email{jiyadani@mit.edu}

\author{Felix Gotti}
\address{Department of Mathematics\\MIT\\Cambridge, MA 02139}
\email{fgotti@mit.edu}

\author{Bryan Li}
\address{Department of Mathematics\\Yale University\\New Haven, CT 06511}
\email{bryan.li.bl846@yale.edu}

\author{Arav Paladiya}
\address{Department of Mathematics\\UC Berkeley\\Berkeley, CA 94720}
\email{aravp@berkeley.edu}

\date{\today}

\begin{abstract}
	We study the additive and multiplicative arithmetic of algebraic monogenic semidomains of the form $S_\alpha$. We first relate their structure to the conjugates of $\alpha$, characterize when $S_\alpha$ is a ring, and determine when its additive monoid is free. We then investigate units and prove a Dirichlet-type theorem showing that $S_\alpha^\times$ is the product of a finite cyclic group of roots of unity and a finitely generated free abelian group; in particular, the unit group is free abelian when $\alpha$ is positive. We also give a criterion that produces non-atomic algebraic monogenic semidomains. Turning to finiteness properties, we prove that $\mathbb{Z}[\alpha]$ is a finite factorization domain for every algebraic number $\alpha$, establish criteria ensuring that $S_\alpha$ is a finite factorization semidomain, and construct infinitely many positive cubic generators yielding finite factorization semidomains that are not rings. Finally, we give a membership criterion for rational elements of $\mathbb{Z}[\beta]$ for arbitrary algebraic $\beta$, and prove that, for algebraic integers $\alpha$, unique factorization ascends from $\mathbb{Z}[\alpha]$ to $\mathbb{Z}[\alpha/n]$ for every $n\in\mathbb{N}$.
\end{abstract}

\bigskip
\maketitle

\section{Introduction}
\label{sec:intro}

Classical factorization theory developed primarily in rings of integers, Dedekind domains, and closely related classes of integral domains. The paper of Anderson, Anderson, and Zafrullah~\cite{AAZ90}, published in 1990, marked the beginning of more than three decades of systematic study of factorization theory in the class of integral domains. During this period, the study of factorization continue expanding to more general areas such as commutative rings with zero-divisors and monoids that are not necessary commutative or cancellative. 
However, in the general setting of semidomains, factorization has been studied only sporadicaly, with in the past three decades. Among the notable exceptions, Cesarz, Chapman, McAdam, and Schaeffer studied elasticity in semirings defined by positive systems, including the polynomial semidomain $\rr_{\ge 0}[x]$ and certain algebraic monogenic semidomains generated by quadratic numbers~\cite{CCMS09}, while Ponomarenko investigated atomicity and finiteness properties in commutative semigroup semirings~\cite{vP15}.
\smallskip

A subsemiring of a field is called monogenic provided that it can be generated by only one element over its prime subsemiring. We let $S_\alpha$ denote the monogenic semiring generated by a complex number $\alpha$ inside $\cc$, that is,
\[
	S_\alpha := \{f(\alpha):f(x)\in\nn_0[x]\},
\]
where $\nn_0[x]$ denotes the polynomial semiring with nonnegative integer coefficients. The simultaneous appearance in 2019 of two complementary papers on monogenic semidomains has motivated an active and systematic investigation of the algebraic, arithmetic, factorization, and ideal-theoretic aspects of such semirings.On the additive side, Chapman, Gotti, and Gotti studied the arithmetic of lengths and other factorization invariants of the additive monoids underlying rational monogenic semidomains~\cite{CGG20}. Their work motivated the systematic study of the additive structure of algebraic monogenic semidomains initiated by Correa-Morris and Gotti~\cite{CG22} and continued from a valuation-theoretic perspective by Chen, Gotti, Lu, and Yao~\cite{CGLY26}. Related contributions include the work of Jiang, Li, and Zhu on primality and elasticity~\cite{JLZ23}, that of Ajran, Bringas, Li, Singer, and Tirador on additive factorizations in evaluation semirings~\cite{ABLST23}, and that most recent of Dani, Deng, Gotti, Li, Paladiya, Vulakh, and Zeng on atoms and strong atoms~\cite{DDHLPV25}.
\smallskip

On the multiplicative side, Campanini and Facchini gave a systematic treatment of factorization and ideal-theoretic properties of the polynomial semiring $\nn_0[x]$~\cite{CF19}. Note that $\nn_0[x]$ is isomorphic to the monogenic semidomain $S_\tau$ for every transcendental number $\tau \in \cc$. Motivated by their work, Gotti and Polo investigated the ascent of atomicity and standard factorization properties to polynomial and Laurent-polynomial semidomains~\cite{GP23}; they subsequently studied the ascent of several weaker, subatomic properties in the same setting~\cite{GPP23}. More recently, Gonzalez, Gotti, and Panpaliya considered the ascent of almost atomicity and quasi-atomicity to polynomial and monoid semidomains~\cite{GGP25}.
\smallskip

The Bi-UF Positive Conjecture, posed by Baeth, Chapman, and Gotti in their study of bi-atomic classes of positive semirings~\cite{BCG21}, has provided another important motivation for the factorization theory of semidomains. The conjecture asserts that $\nn_0$ is the only subsemiring of $\rr_{\ge 0}$ whose additive and multiplicative monoids are both factorial. After several years with little progress, a number of recent results have established the conjecture for increasingly broad classes of monogenic semidomains. Deng, Gotti, and Zeng proved the rational case as part of their classification of factorial rational monogenic semidomains~\cite{DGZ26}. Gotti, Graia, Han, and Liang then settled the case of quadratic algebraic generators~\cite{GGHL26}. Most recently, Bilakanti, Gotti, Kandasamy, Liang, Liu, Polo, Yang, and Yao proved the conjecture for finitely generated algebraic positive semidomains and obtained a corresponding reduction theorem for finitely generated algebraic semidomains in the complex plane~\cite{BGKLPYY26}.
\smallskip

The first systematic study of multiplicative factorizations in rational monogenic semidomains was carried out by Deng, Gotti, and Zeng~\cite{DGZ26}. Among other results, they proved that, for $q\in\qq_{>0}$, the semidomain $S_q$ is factorial if and only if
\[
	q\in\nn\cup\nn^{-1}.
\]
More generally, within the class of rational monogenic semidomains, factoriality, half-factoriality, and the Krull property are equivalent.
\smallskip

The primary purpose of the present paper is to initiate a systematic study of multiplicative factorizations in algebraic monogenic semidomains. We begin with structural results relating $S_\alpha$ to the conjugates of~$\alpha$, characterize when $S_\alpha$ is a ring, and determine when its additive monoid is free. We then study the unit group and prove a Dirichlet-type theorem: for every algebraic number~$\alpha$, the group $S_\alpha^\times$ is the direct product of a finite cyclic group of roots of unity and a finitely generated free abelian group. We next turn to atomicity and give a criterion producing algebraic monogenic semidomains whose multiplicative monoids are not atomic. The analogous question for rational monogenic semidomains was posed in~\cite{DGZ26}.\footnote{According to ongoing work of Cen et al., the analogous rational question has since been answered affirmatively.}
\smallskip

Finally, we investigate finiteness and uniqueness properties of multiplicative factorizations. We prove that $\zz[\alpha]$ is a finite factorization domain for every algebraic number~$\alpha$; consequently, every algebraic monogenic semidomain that is a ring is a finite factorization semidomain. We also establish criteria ensuring that $S_\alpha$ has the finite factorization property and construct infinitely many positive cubic generators for which $S_\alpha$ is a finite factorization semidomain but not a ring. We conclude with a rational-membership test for $\zz[\beta]$ for arbitrary algebraic $\beta$ and prove that, whenever $\alpha$ is an algebraic integer and $\zz[\alpha]$ is a UFD, every $\zz[\alpha/n]$ is a UFD.

\bigskip
\section{Preliminaries}
\label{sec:prelim}

In this section, we collect the notation and terminology used throughout the paper. Our conventions concerning commutative monoids and factorization theory are standard; see~\cite{GH06} for a comprehensive treatment.

\smallskip
\subsection{General Notation}

We let $\pp$ denote the set of rational primes, while we respect all the standard notations $\zz$, $\qq$, $\overline{\mathbb{Z}}$, $\rr$, and $\cc$ to refer to the rational integers, rational numbers, algebraic integers, real numbers, and complex numbers. We also adopt the standard notation $\ff_{q}$ to refer to the finite field of $q$ elements, where $q$ is a prime power. Lastly, it will be notationally convenient for us to let $\aaa$ denote the ring consisting of all algebraic numbers. Also, for $a,b \in \zz$ with $a \le b$, we will adopt the notation
\[
	\ldb a,b\rdb := \{n \in \zz : a \le n \le b \}
\]
for the discrete interval from $a$ to $b$. If $X\subseteq\rr$ and $r \in \rr$ then $X_{\ge r} := \{x \in X : x \ge r\}$ and, in a similar manner, we define the notation $X_{>r}$ and $X_{\le r}$. 
For a prime number $p$ and a nonzero integer $a$, the $p$-adic valuation $\nu_p(a)$ is the largest $e\in\nn_0$ such that $p^e$ divides~$a$.
\smallskip


\medskip
\subsection{Integral Domains, Orders, and Rings of Integers} 

For an integral domain $R$, we denote its field of fractions by $F(R)$ and its group of units by $R^\times$. An \emph{overring} of $R$ is a ring $T$ satisfying $R \subseteq T \subseteq \qf(R)$. 
A \emph{Dedekind domain} is a Noetherian, integrally closed domain in which every nonzero prime ideal is maximal; equivalently, every nonzero fractional ideal is a unique finite product of nonzero prime ideals with integer exponents. 
Following standard notation, we write $\operatorname{Spec}(R)$ for the set of prime ideals of~$R$.

We let $\mu(K)$ denote the set consisting of all roots of unity contained in~$K$. Observe that $\mu(K)$ is a subgroup of $K^\times$. 
Let $K$ be a number field and let $\OK$ denote its ring of integers.  For $\alpha \in K$, we denote the field norm from $K$ to~$\qq$ by $N_{K/\qq}(\alpha)$. If $K$ has $r_1$ real embeddings and $2r_2$ nonreal embeddings into~$\cc$ then $(r_1,r_2)$ is the \emph{signature} of~$K$. 

An \emph{order} in $K$ is a subring~$R$ of $\mathcal O_K$ containing~$1$ such that $R$ is a free $\zz$-module of rank $[K:\qq]$; equivalently, $R$ has finite index in $\mathcal O_K$.
\smallskip

Let $S$ be a finite set of prime ideals of $\OK$. An element $\alpha \in K$ is called an $S$\emph{-unit} if the principal fractional ideal $\alpha \OK$ is in the group generated by the prime ideals in~$S$. For the ring of rational integers $\zz$ one may take $S$ to be a finite set of prime numbers and define an $S$-unit to be a rational number whose numerator and denominator have all prime divisors in~$S$. It is well known that $\OK^\times$ is contained in the group of $S$-units. 
We say that a nonzero prime ideal $\mfp$ of $\OK$ is a \emph{denominator prime} of $\theta \in K^\times$ if $\nu_\mfp (\theta) < 0$. 
\smallskip

An \emph{archimedean place} (or \emph{infinite place}) of $K$ is an equivalence class of embeddings $\sigma \colon K \hookrightarrow \cc$, where two embeddings are identified if they are equal or complex conjugate. A \emph{real place} is represented by a single embedding into $\rr$, whereas a \emph{complex place} consists of a pair of nonreal conjugate embeddings $\sigma$ and $\bar\sigma$, where $\bar\sigma(a)=\overline{\sigma(a)}$ for $a\in K$. Thus, if $K$ has signature $(r_1,r_2)$, then it has exactly $r_1+r_2$ archimedean places. The ordinary absolute value $a\mapsto|\sigma(a)|$ does not depend on the representative of the place. For the product formula, we use the normalized value $|a|_v = |\sigma(a)|$ at a real place and $|a|_v = |\sigma(a)|^2$ at a complex place.
\smallskip

\smallskip
\subsection{Atomicity and Factorization}

Throughout this paper, the monoids considered in factorization theory are assumed to be commutative and cancellative. The multiplicative monoid of a semiring including its zero is exempt from the cancellativity convention. We normally write a monoid multiplicatively, with identity~$1$. Its group of units is denoted by $M^\times$. Two elements $a,b\in M$ are \emph{associates} if $a=ub$ for some $u\in M^\times$. We say that $a$ \emph{divides} $b$ in~$M$, and write $a\mid_M b$, if $b=ac$ for some $c\in M$. The monoid $M$ is \emph{reduced} if $M^\times=\{1\}$. When a monoid is written additively, its identity is denoted by~$0$, its group of units by $\uu(M)$, and the preceding definitions are interpreted additively. An element $a \in M\setminus M^\times$ is an \emph{atom} if $a=bc$ for some $b,c \in M$ implies that either $b$ or $c$ is a unit. The set of atoms of $M$ is denoted by $\mathcal A(M)$. The monoid $M$ is \emph{atomic} if every nonunit is a finite product of atoms, and it is \emph{antimatter} if it has no atoms. An element $p\in M\setminus M^\times$ is \emph{prime} if $p\mid_M ab$ implies that $p\mid_M a$ or $p \mid_M b$.
\smallskip

Let $M_{\mathrm{red}}:=M/M^\times$ be the reduced quotient of~$M$, and let $\mathsf Z(M)$ be the free commutative monoid on $\mathcal A(M_{\mathrm{red}})$. The natural homomorphism $\pi \colon \mathsf Z(M) \rightarrow M_{\mathrm{red}}$ is called the \emph{factorization homomorphism}. For $b \in M$, the set $\mathsf Z_M(b):=\pi^{-1}(bM^\times)$ is the set of factorizations of~$b$. An atomic monoid $M$ is a \emph{unique factorization monoid} (UFM) if $|\mathsf Z_M(b)|=1$ for every $b \in M$, a \emph{finite factorization monoid} (FFM) if $|\mathsf Z_M(b)|<\infty$ for every $b \in M$. A useful characterization that we repeatedly use is that $M$ is an FFM if and only if every element of~$M$ has only finitely many divisors up to associates~\cite{fHK92}.
\smallskip

If $z \in\mathsf Z(M)$ is a product of $\ell$ atoms of $M_{\text{red}}$ (counting repetitions) then $|z| := \ell$ is called the \emph{length} of the factorization~$z$. In addition, for each $b \in M$, the set
\[
	\mathsf L_M(b) := \{|z|:z\in\mathsf Z_M(b)\}
\]
is called the \emph{set of lengths} of~$b$. The monoid $M$ is called a \emph{bounded factorization monoid} (BFM) if $|\mathsf L_M(b)|<\infty$ for every $b \in M$. Both implications in the following diagram
\[
	\textbf{UFM} \ \Longrightarrow \ \textbf{FFM} \ \Longrightarrow \ \textbf{BFM}.
\]
are immediate consequences of the mentioned definitions.

A subset $I\subseteq M$ is an \emph{ideal} of~$M$ if $IM\subseteq I$; because $1 \in M$, this is equivalent to $IM=I$. An ideal is \emph{principal} if it has the form $aM$ for some $a\in M$. We say that $M$ satisfies the \emph{ascending chain condition on principal ideals} (ACCP) if every ascending chain of principal ideals eventually stabilizes. It is no difficult to argue that every BFM satisfies the ACCP and every monoid that satisfies the ACCP is atomic. Every monoid satisfying the ACCP is atomic. The monoid $M$ is said to be a \emph{valuation monoid} if its divisibility relation is total, that is, if $a \mid_M b$ or $b \mid_M a$ for all $a,b \in M$.
\smallskip

The \emph{Grothendieck group} $\gp(M)$ of an additively written monoid $M$ is its universal abelian group of differences. For an abelian group $G$, we set $\rank G:=\dim_\qq(\qq\otimes_\zz G)$. The \emph{rank} of $M$ is
\[
	\rank M := \rank\gp(M).
\]
Finally, a commutative monoid is \emph{free on a set $X$} if the natural homomorphism from the monoid of finitely supported functions $X \to \nn_0$ is an isomorphism; equivalently, every element has a unique expression, up to order, as a finite product of elements of~$X$. For an additively written monoid, products are replaced by sums.

\medskip
\subsection{Polynomials}

Let $R$ be a commutative ring with an identity element and let $R[x]$ denote the polynomial extension of $R$. Then let
\[
	f(x) := \sum_{i=0}^d c_i x^i \in R[x]
\]
be a nonzero polynomial with $c_d \neq 0$. We write $\deg f = d$ and $[x^i]f(x) := c_i$. The \emph{support} of~$f$ is
\[
	\supp f(x):=\{i\in\ldb 0,d\rdb:c_i\neq0\}.
\]
For a nonempty finite subset $E\subseteq\nn_0$, the notation $\gcd E$ means the greatest common divisor of its elements, with the convention $\gcd\{0\}=0$.
\smallskip

For a nonzero polynomial $g(x) \in \zz[x]$, the \emph{content} $\operatorname{cont}(g)$ is the positive greatest common divisor of its coefficients. The polynomial $g$ is \emph{primitive} if $\operatorname{cont}(g) = 1$. We use Gauss's lemma in the following standard form: if $g,h\in\zz[x]$, the polynomial $g$ is primitive, and $g$ divides $h$ in $\qq[x]$, then $g$ divides $h$ in $\zz[x]$.
\smallskip

Let $\alpha\in\aaa$. Its \emph{minimal polynomial} over~$\qq$, denoted by $m_\alpha(x)$, is the unique monic irreducible polynomial in $\qq[x]$ having $\alpha$ as a root. The complex roots of $m_\alpha(x)$ are the \emph{conjugates} of~$\alpha$. The number $\alpha$ is an algebraic integer if and only if $m_\alpha(x)\in\zz[x]$. We denote by $\omega_\alpha(x)$ the unique primitive polynomial in $\zz[x]$ with positive leading coefficient that is a rational multiple of $m_\alpha(x)$. There are unique polynomials $p_\alpha,q_\alpha\in\nn_0[x]$ with disjoint supports such that
\[
	\omega_\alpha(x)=p_\alpha(x)-q_\alpha(x).
\]
The pair $(p_\alpha,q_\alpha)$ is called the \emph{minimal pair} of~$\alpha$.
\smallskip

For a nonzero real polynomial, a \emph{variation of sign} is a sign change between consecutive nonzero coefficients. Descartes' rule of signs states that the number of positive real roots, counted with multiplicity, is at most the number of sign variations and differs from it by an even nonnegative integer.

We also use the following form of P\'olya's theorem. If $F\in\rr[X_1,\dots,X_n]$ is homogeneous and strictly positive on the standard simplex
\[
	\Delta_{n-1}:=\big\{(x_1,\dots,x_n) \in \rr_{\ge0}^n : x_1 + \dots+ x_n = 1 \big\},
\]
then, for all sufficiently large $N \in \nn$, the polynomial
\[
	(X_1 + \dots + X_n)^N F(X_1, \dots, X_n)
\]
has strictly positive coefficients~\cite{gP28}.

\medskip
\subsection{Commutative Semirings}

By a \emph{semiring} we mean a set $S$ equipped with binary operations $+$ and $\cdot$ such that the following conditions hold:
\begin{enumerate}
	\item $(S,+)$ is a monoid with identity~$0$,
	\smallskip
	
	\item $(S,\cdot)$ is a commutative semigroup with identity~$1$, and
	\smallskip
	
	\item multiplication distributes over addition. 
\end{enumerate}
Thus, all semirings in this paper are commutative and have an identity. We write $rs$ instead of $r \cdot s$. A semiring is a \emph{ring} if every element has an additive inverse. An \emph{integral domain} is a nonzero commutative ring with no nonzero zero divisors.
\smallskip

The \emph{characteristic} of a semiring $S$, denoted by $\operatorname{char}S$, is the least positive integer $n$ such that $n1=0$, if such an integer exists, and is~$0$ otherwise. The image of the canonical homomorphism $\nn_0\to S$, given by $n\mapsto n1$, is the \emph{prime subsemiring} of~$S$ and is denoted by~$\sss$.

A subset $T\subseteq S$ is a \emph{subsemiring} if $0,1\in T$ and $T$ is closed under addition and multiplication. A map $\varphi\colon S\to T$ between semirings is a \emph{semiring homomorphism} if it preserves $0$, $1$, addition, and multiplication. A bijective semiring homomorphism is a semiring \emph{isomorphism}. An \emph{ideal} of $S$ is an additive submonoid $I\subseteq S$ such that $SI\subseteq I$; the principal ideal generated by $a\in S$ is $aS$.
\smallskip

A semiring $S$ is a \emph{semidomain} if it is isomorphic to a subsemiring of an integral domain. In particular, a semidomain is additively cancellative, satisfies $1\neq0$, and has no nonzero zero divisors. The Grothendieck group of the additive monoid $(S,+)$ carries a unique multiplication extending that of~$S$; the resulting integral domain is denoted by $\mathcal D(S)$ and is called the \emph{Grothendieck domain} of~$S$. The inclusion $S\hookrightarrow\mathcal D(S)$ is universal among semiring homomorphisms from $S$ to integral domains. The \emph{field of fractions} of~$S$ is the field $\qf(\mathcal D(S))$.
\smallskip

For a semidomain $S$, we set $S^* := S\setminus\{0\}$ and regard $S^*$ as a multiplicative monoid. The units of $S$ form the group $S^\times=(S^*)^\times$. We call $S$ \emph{atomic}, a \emph{unique factorization semidomain} (UFS), a \emph{finite factorization semidomain} (FFS), or a \emph{bounded factorization semidomain} (BFS) when $S^*$ has the corresponding monoid property. If $S$ is an integral domain, we use the traditional terms \emph{unique factorization domain} (UFD), \emph{finite factorization domain} (FFD), and \emph{bounded factorization domain} (BFD).
\smallskip

The characteristic of a semidomain agrees with that of its Grothendieck domain and is therefore either $0$ or a prime number. Its prime subsemiring is isomorphic to $\nn_0$ in characteristic~$0$ and to $\ff_p$ in characteristic~$p$.

\begin{definition}
	A semiring $S$ is \emph{monogenic} if there exists $\alpha\in S$ such that $S=\sss[\alpha]$, where $\sss$ is the prime subsemiring of~$S$. The element $\alpha$ is then called a \emph{generator} of~$S$.
\end{definition}

If $R$ is a subsemiring of a commutative ring and $\alpha$ lies in a ring extension of~$R$, then
\[
	R[\alpha] := \{f(\alpha) : f(x) \in R[x]\}.
\]
In particular, for $\alpha \in \cc$ we set $S_\alpha := \{f(\alpha) : f(x) \in \nn_0[x]\}$.

Let $S = \ff_p[\alpha]$ be a monogenic semidomain of characteristic $p \in \pp$. If $\alpha$ is algebraic over $\ff_p$, then $S$ is a finite integral domain and hence a field. If $\alpha$ is transcendental over $\ff_p$, then $S \cong \ff_p[x]$. Therefore every monogenic semidomain of positive characteristic is a UFD. Accordingly, we restrict our attention to characteristic zero.

\begin{prop} \label{prop:monogenic semidomain of char 0 are realizable in the complex plane}
	Let $S$ be a monogenic semidomain of characteristic zero. Then there exists $\alpha \in \cc$ such that $S$ is isomorphic to the semidomain $S_\alpha$.
\end{prop}

\begin{proof}
	Since $S$ has characteristic zero, its prime subsemiring is isomorphic to $\nn_0$. As $S$ is monogenic, there exists $\sigma \in S$ such that $S=S_\sigma$. View $S$ inside its Grothendieck domain $\mathcal{D}(S)$, and let $K$ be the field of fractions of $\mathcal{D}(S)$. Then $\qq(\sigma)$ is a subfield of $K$.

	Suppose first that $\sigma$ is algebraic over $\qq$. Choose a complex root $\alpha$ of the minimal polynomial of $\sigma$ over $\qq$. Then there is a field isomorphism $\varphi \colon \qq(\sigma) \longrightarrow \qq(\alpha)$ that fixes $\qq$ and maps $\sigma$ to $\alpha$. Its restriction to $S=S_\sigma$ is therefore a semiring isomorphism from $S$ onto $S_\alpha$.

	If $\sigma$ is transcendental over $\qq$, choose a transcendental number $\tau \in \cc$. The assignment $\sigma \mapsto \tau$ induces a field isomorphism $\qq(\sigma) \to \qq(\tau)$ fixing~$\qq$, whose restriction to $S$ is a semiring isomorphism from $S$ onto $S_\tau$.
\end{proof}

Proposition~\ref{prop:monogenic semidomain of char 0 are realizable in the complex plane} allows us to study monogenic semidomains of characteristic zero inside the complex plane. If $\tau\in\cc$ is transcendental, then evaluation at~$\tau$ gives an isomorphism $\nn_0[x]\cong S_\tau$. Since the factorization and ideal theory of $\nn_0[x]$ was studied by Campanini and Facchini~\cite{CF19}, we henceforth focus on semidomains $S_\alpha$ generated by algebraic numbers $\alpha\in\cc$.

We conclude with a characterization of those algebraic generators for which $S_\alpha$ is a ring; see also~\cite{GHLL26}.

\begin{prop}\label{prop:monogenic semidomains that are ID}
	For $\alpha\in\aaa$, the following conditions are equivalent.
	\begin{enumerate}
		\item[(a)] $S_\alpha$ is a ring, in which case $S_\alpha=\zz[\alpha]$ is an integral domain.
		\smallskip

		\item[(b)] No conjugate of $\alpha$ belongs to $\rr_{\ge0}$.
	\end{enumerate}
\end{prop}

\begin{proof}
	Suppose first that $S_\alpha$ is a ring. Then $-1\in S_\alpha$, so there exists $f(x)\in\nn_0[x]$ such that $f(\alpha)=-1$. Hence the polynomial $h(x):=f(x)+1$ belongs to $\nn_0[x]$, has positive constant term, and satisfies $h(\alpha)=0$. Every conjugate $\beta$ of~$\alpha$ is also a root of~$h$. If $\beta\in\rr_{\ge0}$, then $h(\beta)>0$, a contradiction. Thus, no conjugate of~$\alpha$ is nonnegative real.

	Conversely, suppose that no conjugate of $\alpha$ belongs to $\rr_{\ge0}$. Let $\omega_\alpha(x)\in\zz[x]$ be the primitive integer multiple of $m_\alpha(x)$ with positive leading coefficient, and set $d:=\deg\omega_\alpha$. Since $\omega_\alpha$ has no roots in $\rr_{\ge0}$ and has positive leading coefficient, we have $\omega_\alpha(r)>0$ for every $r\in\rr_{\ge0}$. Consider its homogenization
	\[
		\Omega(X,Y):=Y^d\omega_\alpha\bigg( \frac{X}Y\bigg).
	\]
	The polynomial $\Omega$ is strictly positive on the simplex $\{(X,Y)\in\rr_{\ge0}^2:X+Y=1\}$. By P\'olya's theorem, there exists $N\in\nn_0$ such that
	\[
		H(X,Y):=(X+Y)^N\Omega(X,Y)
	\]
	has nonnegative integer coefficients. Setting $Y=1$ gives
	\[
		h(x):=H(x,1)=(x+1)^N\omega_\alpha(x) \in \nn_0[x].
	\]
	Since $h(\alpha)=0$ and $h(0)>0$, we obtain that $-h(0) = (h-h(0))(\alpha)\in S_\alpha$, and so $-1 \in S_\alpha$ because we can write $-1 = -h(0)+(h(0)-1)$. This implies that $S_\alpha$ is a ring. Because $S_\alpha$ is a semidomain, it is then an integral domain, and the inclusion $S_\alpha\subseteq\zz[\alpha]$ is immediate, while the fact that $S_\alpha$ is a ring gives the reverse inclusion. Hence $S_\alpha=\zz[\alpha]$.
\end{proof}

\bigskip
\section{A Few Words on the Additive Structure}
\label{sec:additive structure}

In this section, we consider the underlying additive monoid of monogenic semidomains, and we determine when they are free commutative monoids. We proceed to answer the question of whether the additive monoid of a monogenic semiring of characteristic zero is a free commutative monoid.

\begin{prop}
    Let $S$ be a monogenic semidomain of characteristic zero, identify $S$ with $S_\alpha$ for a chosen generator $\alpha \in \cc$. Theb the following conditions are equivalent.
    \begin{enumerate}
        \item[(a)] $\alpha$ is transcendental.
		\smallskip

        \item[(b)] $(S\,+)\cong(\nn_0[x],+)$.
		\smallskip

        \item[(c)] Every element of $S$ has a unique expression $\sum_{n\ge0}a_n\alpha^n$, where $(a_n)_{n\ge0}$ is a finitely supported family in $\nn_0$.
    \end{enumerate}
\end{prop}

\begin{proof}
    (a) $\Rightarrow$ [(b),(c)]: Just notice that evaluation at $\alpha$ is injective on $\nn_0[x]$ when $\alpha$ is transcendental.
	\smallskip

	(b) $\Rightarrow$ (a): The difference group of $(S_\alpha,+)$ is $\zz[\alpha]$. If $\alpha$ is algebraic, this group has finite rank $[\qq(\alpha):\qq]$, whereas the difference group of $(\nn_0[x],+)$ has infinite rank.
	\smallskip

    (c) $\Rightarrow$ (a): Finally, if $\alpha$ is algebraic, split its nonzero primitive integer polynomial as $\omega_\alpha = p-q$ with $p,q \in \nn_0[x]$ and disjoint supports. The distinct coefficient families of $p$ and $q$ give $p(\alpha) = q(\alpha)$, contradicting condition~(c).
\end{proof}

\begin{prop}\label{prop:algebraic-additive-free}
    Let $\alpha \in \aaa$ have degree $d$. Then $(S_\alpha,+)$ is free if and only if
    \[
        m_\alpha(x)=x^d-\sum_{i=0}^{d-1}c_i x^i
    \]
    for some coefficients $c_0, \dots, c_{d-1} \in \nn_0$, in which case, $(S_\alpha,+)$ is free on $\{1,\alpha, \dots, \alpha^{d-1}\}$.
\end{prop}

\begin{proof}
    If the displayed identity holds, it reduces every power of $\alpha$ to a nonnegative integer combination of the first $d$ powers. Their linear independence over $\qq$ gives uniqueness.

    Conversely, suppose $(S_\alpha,+)$ is free. The case $\alpha=0$ is immediate. A nonzero free monoid is reduced and cannot be an abelian group, so Proposition~\ref{prop:monogenic semidomains that are ID} allows us to replace $\alpha$ by a positive conjugate. The case $\alpha=1$ is immediate as well. Otherwise its powers are distinct. Every additive atom is a power of $\alpha$, because these powers generate $S_\alpha$. Also, $1$ is an atom: a decomposition of $1$ into two nonzero elements would, upon multiplication, decompose every nonzero element.

    The difference group of $S_\alpha$ has rank $d$, so a free additive basis has exactly $d$ atoms. Choose the least $n$ such that $\alpha^n$ is not an atom. Every later power is also a nonatom, by multiplying a nontrivial decomposition of $\alpha^n$ by a positive power of $\alpha$. Consequently the atoms are exactly $1,\alpha,\ldots,\alpha^{n-1}$ and $n=d$. Expressing $\alpha^d$ in this basis gives the displayed monic degree-$d$ polynomial, which must be $m_\alpha$.
\end{proof}

\bigskip
\section{The Group of Units}

This section is devoted to the group of units of $S_\alpha$, where $\alpha \in \aaa$. Observe that when $\alpha \in \nn$ the equality $S_\alpha = \nn_0$ holds, and so $S_\alpha^\times$ is the trivial group.

\begin{prop}\label{prop:nontrivial units when N_0[alpha] is not an ID}
	Let $\alpha$ be a positive algebraic number such that $\alpha \notin \nn$. Then $\alpha \in S_\alpha^\times$ if and only if $S_\alpha^\times$ is nontrivial.
\end{prop}

\begin{proof}
	Set $S := S_\alpha$. If $\alpha \in S^\times$, then $S^\times$ is nontrivial because $\alpha \notin \nn$ and, in particular, $\alpha \neq 1$.
	\smallskip
	
	Conversely, suppose that $S^\times$ is nontrivial, and fix $u \in S^\times \setminus \{1\}$. Then we can take polynomials $f(x), g(x) \in \nn_0[x]$ such that $f(\alpha) = u$ and $g(\alpha) = u^{-1}$. Since $u \neq 1$, neither $f(\alpha)$ nor $g(\alpha)$ is equal to~$1$. Also, as $\alpha > 0$ and $f(\alpha)g(\alpha) = 1$, both $f(\alpha)$ and $g(\alpha)$ are positive, and so $\min\{f(\alpha), g(\alpha)\} < 1$. If $f(\alpha) < 1$, then $0 \le f(0) \le f(\alpha) < 1$, which gives $f(0) = 0$. Similarly, $g(0) = 0$ whenever $g(\alpha) < 1$. Hence $(fg)(0) = 0$, so $fg = xh$ for some $h(x) \in \nn_0[x]$. Therefore
	\[
		1 = f(\alpha)g(\alpha) = \alpha h(\alpha),
	\]
	and $h(\alpha) \in S$. Thus, $\alpha$ is a unit of $S$, as desired.
\end{proof}

If the generator $\alpha$ is a nontrivial unit then the underlying additive monoid of $S_\alpha$ contains no additive atoms.

\begin{prop}\label{prop:nontrivial unit group additive antimatter}
	For an algebraic number $\alpha \in \aaa$, the following statements hold.
	\begin{enumerate}
		\item[(1)] If $\alpha \in S_\alpha^\times \setminus \{1\}$ then $S_\alpha$ is additively antimatter. 
		\smallskip
		
		\item[(2)] If $\alpha$ has a positive conjugate then the converse is also true.
	\end{enumerate}
\end{prop}

\begin{proof}
	(1) Fix $\alpha \in \aaa$ and assume that $\alpha$ is a unit of $S_\alpha$ with $\alpha \neq 1$. If $S_\alpha$ is an integral domain then the additive monoid of $S_\alpha$ is an abelian group and so it is trivially antimatter. Therefore we assume that $S_\alpha$ is not an integral domain. As additive antimatterness is preserved under isomorphism, after replacing $\alpha$ by one of its positive conjugates we can further assume that $S_\alpha \subseteq \rr_{\ge 0}$. Hence $(S_\alpha,+)$ is a  reduced monoid. As $\alpha \in S_\alpha^\times$, we can take $f(x) \in \nn_0[x]$ such that $\alpha f(\alpha) = 1$. Note the this equality yields a nontrivial additive decomposition of $1$ because $f(x)$ is not a monic monomial.
	
	From the fact that $\alpha \in \rr_{> 0} \setminus \{1\}$ one deduces that $f(x)$ is not a monic monomial. This, along with the fact that $(S_\alpha, +)$ is reduced, implies that $1$ is not an additive atom and so, after taking nonzero $\sigma_1,\sigma_2 \in S_\alpha$ such that $1=\sigma_1+\sigma_2$, we see that $\beta = \beta \sigma_1 + \beta\sigma_2$ for every nonzero $\beta\in S_\alpha$. Both summands are nonzero because the semidomain has no zero-divisors. Hence $S_\alpha$ is additively antimatter.
	\smallskip

    (2) Now suppose $\alpha$ has a positive conjugate and $S_\alpha$ is additively antimatter. Since $S_\alpha$ and $S_\beta$ are isomorphic for conjugates $\alpha$ and $\beta$, we may assume without loss of generality that $\alpha$ is positive. Then there exist $f(x), g(x)\in\nn_0[x]$ such that $f(\alpha),g(\alpha)>0$ and $f(\alpha) + g(\alpha)=1$. Both values are less than~$1$, so both constant terms are zero. Writing $f=xf_1$ and $g=xg_1$ gives $\alpha(f_1(\alpha) + g_1(\alpha)) = 1$. Also, $\alpha \ne 1$, since $(S_1,+) = (\nn_0,+)$ has the atom $1$. Thus, $\alpha\in S_\alpha^\times\setminus\{1\}$.
\end{proof}

We have seen in Proposition~\ref{prop:nontrivial units when N_0[alpha] is not an ID} that if $S_\alpha$ is not an integral domain then either $S_\alpha^\times$ is trivial or $\alpha$ is a unit. Let us now look at the case when $S_\alpha$ is an integral domain and, in this case, determine what are the conditions under which $\alpha$ is a unit of $S_\alpha$.  For $\alpha \in \aaa$, recall that $\omega_\alpha(x) \in \zz[x]$ is the unique primitive integer multiple of $m_\alpha(x)$ with positive leading coefficient.

\begin{prop}\label{prop:when alpha is a unit if N_0[alpha] is an ID}
	Let $\alpha$ be an algebraic number such that $S_\alpha$ is an integral domain. Then $\alpha$ is a unit in $S_\alpha$ if and only if $|\omega_\alpha(0)| = 1$.
\end{prop}

\begin{proof}
	Take $\alpha \in \aaa$ such that $S_\alpha$ is an integral domain, which means that $S_\alpha = \zz[\alpha]$.
	
	For the direct implication, assume that $\alpha \in S_\alpha^\times$. Then we can pick $f(x) \in \nn_0[x]$ such that $\alpha f(\alpha) = 1$. Now set $g(x) := xf(x) - 1 \in \zz[x]$. As $g(\alpha) = 0$, we can write $g(x) = \omega_\alpha(x)h(x)$ for some polynomial $h(x) \in \qq[x]$. As $\omega_\alpha(x)$ has content $1$, Gauss's lemma ensures that $h(x) \in \zz[x]$. Hence $\omega_\alpha(0) h(0) = g(0) = -1$, which implies that $|\omega_\alpha(0)| = 1$.
	
	Conversely, assume that $|\omega_\alpha(0)| = 1$. Because $\omega_\alpha(x) - \omega_\alpha(0) \in x \zz[x]$, we can pick a polynomial $g(x) \in \zz[x]$ such that $\omega_\alpha(x) - \omega_\alpha(0) = xg(x)$. Then $\alpha g(\alpha) = - \omega_\alpha(0) = \pm 1$. Therefore either $\alpha g(\alpha) = 1$ or $\alpha (-g(\alpha)) = 1$. As $g(x) \in \zz[x]$, it follows that $\pm g(\alpha) \in \zz[\alpha] = S_\alpha$, where the last equality holds because $S_\alpha$ is an integral domain. Hence either $-g(\alpha)$ or $g(\alpha)$ is the inverse of $\alpha$ in $S_\alpha$. Thus, $\alpha \in S_\alpha^\times$.
\end{proof}
	
Let $K$ be a number field. Recall that for a finite set of nonzero prime ideals of $\OK$, an element $\alpha \in K^\times$ is an $S$-unit if the principal fractional ideal $\alpha\OK$ belongs to the group generated by the ideals in $S$. For the ring of rational integers $\zz$ one may take $S$ to be a finite set of prime numbers and define an $S$-unit to be a rational number whose numerator and denominator have all prime divisors in~$S$. It is well known that $\OK^\times$ is contained in the group of $S$-units. In addition, the following theorem holds.
	
\begin{theorem}[Dirichlet's $S$-unit Theorem]\label{thm:s unit theorem}
	The group of $S$-units is finitely generated, with rank (maximal number of multiplicatively independent elements) equal to $r+s$, where $r := \rank \, \OK^\times$ and $s = |S|$. More precisely, let~$\mathcal V$ consist of the archimedean places of~$K$ and the primes in~$S$, with the standard product-formula absolute values. Then
	\[
		\big\{(\log|u|_v)_{v\in\mathcal V}:u\text{ is an $S$-unit}\big\}
	\]
	is a full lattice in $\{(x_v)\in\rr^{\mathcal V}:\sum_vx_v=0\}$; that is, it is discrete and spans this hyperplane.
\end{theorem}
	
As the last result of this section, we provide a description of the group of units of $S_\alpha$ for any choice of the algebraic generator~$\alpha$. 
	
\begin{theorem}[Generalized Dirichlet's Unit Theorem] \label{thm:group of unity main}
	For any $\alpha \in \aaa$, the group of units of the semidomain $S_\alpha$ has the form $\mu(S_\alpha) \times \zz^r$, where $r \in \nn_0$ and $\mu(S_\alpha)$ is the cyclic group consisting of all the roots of unity inside $S_\alpha$.
\end{theorem}

\begin{proof}
	The main idea of this proof is to embed the unit group of $S_\alpha$ into the unit group of a specific ring of $S$-units within a number field so that we can then apply Dirichlet's $S$-unit theorem for $S$-units to obtain the desired result.
	\smallskip
	
	Write $\alpha = \omega/d$ with $\omega$ an algebraic integer and $d\in\nn$. Put $K = \qq(\alpha)$ and let $S$ be the finite set of nonzero prime ideals of $\mathcal O_K$ dividing $d \mathcal O_K$. Define
	\[
		\mathcal O_{K,S}:=\{0\} \cup
		\{z \in K^\times : \nu_{\mathfrak p}(z) \ge 0
		\text{ for every nonzero prime }\mathfrak p \notin S \}.
	\]
	Since $d$ has valuation zero at every nonzero prime outside $S$,
	\[
		S_\alpha \subseteq \zz[\alpha] \subseteq \mathcal O_K[1/d] = \mathcal O_{K,S}.
	\]
	In particular, $S_\alpha^\times$ is a subgroup of the finitely generated group
	\[
		\mathcal O_{K,S}^\times \cong \mu(K)\times
		\zz^{r_1+r_2+|S|-1}.
	\]
	The structure theorem for finitely generated abelian groups therefore gives
	\begin{align}\label{eq:subgroup of OKS I}
		S_\alpha^\times\cong C_m\times\zz^n
	\end{align}
	for some $(m,n) \in \nn \times \nn_0$, where $C_m$ is the torsion subgroup of $S_\alpha^\times$:
    \[
        C_m = \mu(S_\alpha) := \{ \zeta \in S_\alpha^\times : \zeta^k = 1 \text{ for some } k \in \nn \} = S_\alpha^\times \cap \mu(K),
    \]
    because the torsion elements of $K^\times$ are precisely the roots of unity in $K$.
\end{proof}
	
As a consequence of Theorem~\ref{thm:group of unity main}, we obtain the following corollary.
	
\begin{cor}
    If $\alpha$ is a positive algebraic number then the group of units of $S_\alpha$ is a finite-rank free abelian group.
\end{cor}

\begin{proof}
    This follows immediately from Theorem~\ref{thm:group of unity main} because if $\alpha$ is a positive algebraic number then the only root of unity that is contained in $S_\alpha$ is~$1$ and so $\mu(S_\alpha)$ is the trivial group.
\end{proof}

We conclude this section with examples.

\begin{example}
	Consider the polynomial $m(x) = x^2 + 3x + 1$, which has two negative real roots $\alpha = (-3 - \sqrt{5})/2$ and $(-3 + \sqrt{5})/2$. Since $m(x)$ is irreducible in $\qq[x]$, it is the minimal polynomial of $\alpha$ and so $\alpha$ is an algebraic integer. As the minimal polynomial has no positive roots, it follows from Proposition~\ref{prop:monogenic semidomains that are ID} that the semiring $S_\alpha$ is an integral domain, which means that $S_\alpha=\zz[\alpha]=\zz[(1+\sqrt5)/2]=\mathcal O_{\qq(\sqrt5)}$. Indeed, with $\varphi=(1+\sqrt5)/2$, we have $\alpha=-1-\varphi$ and $\varphi=-1-\alpha$.

    Since $m(\alpha) = 0$, we see that $\alpha(-\alpha-3) = 1$. Also $N_{\qq(\sqrt5)/\qq}(\alpha) = 1$, the constant term of its monic quadratic minimal polynomial. Hence $\alpha\in\zz[\alpha]^\times = S_\alpha^\times$.
    
    Since $\qq(\alpha)$ is a real quadratic field, the only roots of unity inside $\zz[\alpha]$ are $\{\pm 1\}$. As $\zz[\alpha]$ is the full ring of integers of $\qq(\sqrt5)$, it follows from Dirichlet's Unit Theorem that $S_\alpha^\times = \mathbb{Z}[\alpha]^\times \cong C_2 \times \mathbb{Z}$. \hfill $\blacksquare$
\end{example}

\begin{example}\label{ex:units of S alpha sqrt2}
	Let $\alpha$ denote the algebraic number $1/\sqrt{2}$ and let $S_\alpha$ be the monogenic semidomain generated by $\alpha$. Then $\alpha^2 = 1/2$, and so every element of $S_\alpha$ can be written uniquely as $d_1 + d_2 \alpha$ for nonnegative dyadic rationals, whence
	\[
		S_\alpha = S_{2^{-1}} \oplus S_{2^{-1}} \alpha.
	\]
	Since $\alpha>0$, we see that $S_\alpha\subseteq\rr_{\ge0}$ and, therefore,  $S_\alpha$ is not a ring. On the other hand, the identity $\alpha(2\alpha) = 1$ shows that $\alpha \in S_\alpha^\times$ and $\alpha^{-1} = 2\alpha$. Hence every integral power of~$\alpha$ is a unit.

	Conversely, take $\omega := d_1 + d_2 \alpha \in S_\alpha^\times$, and write its inverse as $\omega^{-1} = e_1 + e_2 \alpha$, where $d_1, d_2, e_1, e_2 \in S_{2^{-1}}$. Since $\alpha^2 = 1/2$, we obtain
	\[
		1 = (d_1 + d_2\alpha)(e_1 + e_2 \alpha)
		  = d_1e_1 + \frac{d_2e_2}{2} + (d_1e_2 + d_2e_1)\alpha.
	\]
	Since $\alpha$ is irrational, both equations $2d_1e_1+d_2e_2=2$ and $d_1e_2+d_2e_1 = 0$ hold. 
	As all four coefficients are nonnegative, either $d_2 = e_2 = 0$ or $d_1 = e_1 = 0$. In the first case, $d_1e_1=1$, while in the second case, $d_2 e_2 = 2$. Because a positive dyadic rational has a dyadic inverse precisely when it is an integral power of~$2$, we obtain $\omega = 2^k = \alpha^{-2k}$ or $\omega = 2^k\alpha = \alpha^{1-2k}$ for some $k \in \zz$. Hence
	\[
		S_\alpha^\times = \langle \alpha \rangle = \{\alpha^n : n \in \zz \} \cong \zz.  
	\] \hfill $\blacksquare$
\end{example}

\bigskip
\section{Atomicity}

The primary goal of this section is to construct a monogenic semiring that is not atomic. We say that a nonconstant polynomial $f(x) \in \qq[x]$ is simple if the only pair $(n, g(x)) \in \nn \times \qq[x]$ that satisfies the equality $f(x) = g(x^n)$ is $(1, f(x))$, which means that the GCD of the support of a simple polynomial must be $1$ (in $\mathbb{N}$). 
For $\alpha\in\aaa$, write $M_\alpha:=(S_\alpha,+)$, set $g=\gcd\supp m_\alpha$, and define the simplified polynomial $h$ by $m_\alpha(x)=h(x^g)$. The polynomial $h$ is irreducible and is the minimal polynomial of $\alpha^g$. Its support has greatest common divisor $1$. We call $M_{\alpha^g}$ the simplified monoid; grouping powers modulo $g$ and using linear independence over $\qq(\alpha^g)$ gives
\[
    M_\alpha\cong M_{\alpha^g}^{\,g}.
\]
For $\alpha=0$, interpret this as the tautological case $g=1$.

\begin{prop}\label{prop: accp counterexample}
	Let $\beta$ be a positive algebraic integer with minimal polynomial $m_\beta(x)$, and set
	\[
		g:=\gcd \emph{supp} \, m_\beta(x).
	\]
	Suppose that $g$ is even, that the additive monoid $(S_\beta,+)$ is antimatter, and that $\beta$ has a conjugate $\rho$ such that $\rho^g \notin \rr$. Then the multiplicative monoid $S_\beta \setminus \{0\}$ is not atomic.
\end{prop}
	
\begin{proof}
	Write $g=2\ell$, and set $\alpha:=\beta^g$ and $\gamma:=\beta^\ell$ so that $\alpha=\gamma^2$. Since $(S_\beta,+)$ is antimatter, we must have $\beta<1$: otherwise, every nonzero element of $S_\beta$ would be at least~$1$, making $1$ an additive atom. Thus, $\alpha,\gamma \in (0,1)$.

	Because every exponent in $\supp m_\beta(x)$ is divisible by $g$, there is a monic polynomial $h(x) \in \zz[x]$ such that $m_\beta(x)=h(x^g)$. The irreducibility of $m_\beta(x)$ implies that $h(x)$ is irreducible, and therefore
	\[
		m_\beta(x)=m_\alpha(x^g).
	\]
	Moreover, $m_\alpha(x^2)$ is irreducible: a nontrivial factorization would remain nontrivial after substituting $x^\ell$ for $x$ and would factor $m_\beta(x)$. Hence $m_\gamma(x)=m_\alpha(x^2)$. In particular,
	\[
		[\qq(\beta):\qq(\alpha)] = g \qquad \text{and} \qquad [\qq(\gamma):\qq(\alpha)] = 2.
	\]

	Set $M:=S_\alpha$. The preceding direct-product decomposition identifies $(S_\beta,+)$ with $(M,+)^g$, so $M$ is antimatter. Moreover, $m_\alpha$ is simple because $g$ is the greatest common divisor of the support of $m_\beta$. Thus \cite[Theorem~5.6]{CGLY26} shows that $(M,+)$ is a valuation monoid. 
 	Since its difference group is $\zz[\alpha]$ and $M \subseteq \rr_{\ge 0}$, we see that $M = \zz[\alpha] \cap \rr_{\ge 0}$. Also, grouping powers of $\beta$ according to their residue classes modulo $g$ shows that every element of $S_\beta$ can be written as
	\begin{equation}\label{eq:beta-normal-form}
		\sum_{i=0}^{g-1} c_i\beta^i
	\end{equation}
	for some $c_0,\dots,c_{g-1}\in M$. This representation is unique because $1,\beta,\dots,\beta^{g-1}$ are linearly independent over $\qq(\alpha)$. Similarly,
	\[
		\zz[\gamma]=\zz[\alpha]\oplus\gamma\zz[\alpha].
	\]
	Since $M$ is antimatter, $1$ is not an additive atom. Choose nonzero $y,z\in M$ such that $1=y+z$, and represent them by polynomials in $\nn_0[x]$. Positivity forces both constant coefficients to be zero, and hence $1=\alpha q(\alpha)$ for some $q(x)\in\nn_0[x]$. Thus, $\alpha^{-1}\in M$. In particular,
	\[
		\gamma^{-1}=\gamma\alpha^{-1}
		\qquad \text{and} \qquad
		\beta^{-1}=\beta^{g-1}\alpha^{-1}
	\]
	belong to $S_\beta$, so $\alpha$, $\gamma$, and $\beta$ are units.

	We next construct sufficiently many units in $\zz[\gamma]$. Let $p$ and $q$ denote the numbers of positive and negative real conjugates of $\alpha$, respectively, and let $s$ denote the number of pairs of nonreal conjugates. The hypothesis $\rho^g\notin\rr$ gives $s\ge1$, while $\alpha>0$ gives $p\ge1$. From $m_\gamma(x)=m_\alpha(x^2)$, we obtain
	\[
		r_1(\gamma)=2p
		\qquad \text{and} \qquad
		r_2(\gamma)=q+2s.
	\]
	Dirichlet's unit theorem for orders now yields
	\begin{equation}\label{eq:unit-rank-gap}
		\rank\zz[\gamma]^\times-\rank\zz[\alpha]^\times=p+s\ge2.
	\end{equation}

	Let $\sigma$ be the $\qq(\alpha)$-automorphism of $\qq(\gamma)$ determined by $\sigma(\gamma)=-\gamma$, and set
	\[
		U:=\big\{\varepsilon\in\zz[\gamma]^\times:\varepsilon>0 \text{ and } \sigma(\varepsilon)>0\big\}.
	\]
	Define $\Delta\colon U\to\rr$ by $\Delta(\varepsilon):=\log\varepsilon-\log\sigma(\varepsilon)$. Since $U$ contains the square of every unit, $\rank U=\rank\zz[\gamma]^\times$. If $\varepsilon=u+v\gamma\in\ker\Delta$, where $u,v\in\zz[\alpha]$, then $\varepsilon=\sigma(\varepsilon)$ and hence $v=0$. Applying the same argument to $\varepsilon^{-1}$ shows that $\varepsilon \in \zz[\alpha]^\times$. Therefore it follows from~\eqref{eq:unit-rank-gap} that
	\[
		\rank\Delta(U)\ge\rank\zz[\gamma]^\times-\rank\zz[\alpha]^\times \ge 2.
	\]
	A noncyclic subgroup of $\rr$ is dense, so $\Delta(U)$ is dense in~$\rr$. Before being able to complete the proof, we need to establish the following claim.
	\smallskip
	
	\noindent \textsc{Claim.} $a + b\gamma \notin \mathcal{A}(S_\beta)$ for any nonzero $a,b \in M$.
	\smallskip
	
	\noindent \textsc{Proof of Claim.} Fix $R := a + b\gamma$, where $a,b \in M \setminus \{0\}$. First, let us verify that $R$ is not a unit: 
 	indeed,
	\[
		R^{-1} = \frac{a-b\gamma}{a^2-\alpha b^2},
	\]
	and the two coefficients in the unique decomposition over $\qq(\alpha)$ have opposite signs, whence the normal form \eqref{eq:beta-normal-form} prevents $R^{-1}$ from belonging to $S_\beta$. By the density of $\Delta(U)$, choose $\varepsilon := u + v\gamma \in U$ such that $\Delta(\varepsilon) > 0$ is sufficiently small. Then $u,v>0$, and so $u,v \in M$. Set $t := v/u$. From
	\[
		e^{\Delta(\varepsilon)} = \frac{1+t\gamma}{1-t\gamma},
	\]
	we see that $t$ can be chosen arbitrarily small. In particular, we may assume that
	\[
		0 < t < \min\left\{\frac{a}{\alpha b},\frac{b}{a}\right\}.
	\]
	Let $\delta := \varepsilon \sigma(\varepsilon) = u^2 - \alpha v^2$. This element is fixed by $\sigma$ and is a product of units, so $\delta$ is a positive unit of $\zz[\alpha]$. Therefore $\delta^{-1} \in M$. Consequently,
	\[
		\varepsilon^{-1}R = \delta^{-1} \big( (ua-\alpha vb)+(ub-va) \gamma \big)
	\]
	has two nonzero coefficients in~$M$. Thus, both $\varepsilon$ and $\varepsilon^{-1}R$ are nonunits of $S_\beta$, and $R = \varepsilon(\varepsilon^{-1}R)$ is a nontrivial factorization. Hence no element $a+b\gamma$ with $a,b\in M\setminus\{0\}$ is an atom.

	Finally, consider $A := 1+\gamma$. We claim that $A$ has no factorization into atoms. Suppose that $A = rs$ with nonzero $r,s \in S_\beta$, and write their normal forms as
	\[
		r=\sum_{i=0}^{g-1}r_i\beta^i
		\qquad \text{and} \qquad
		s=\sum_{j=0}^{g-1}s_j\beta^j,
	\]
	where all coefficients belong to $M$. Since these coefficients are nonnegative, no cancellation is possible in the product. The normal form of $A$ is supported only in the residue classes $0$ and $\ell$ modulo $g$. Fixing any $j$ with $s_j\neq0$, we see that the support of $r$ contains at most two residue classes; if it contains two, their difference is~$\ell$. Therefore $r \in \beta^i S_\gamma$ for some $i \in \ldb 0, \ell-1 \rdb$. The same conclusion holds for every divisor of~$A$. Now let $r$ be an atom dividing $A$. By the preceding support argument, we can write
	\[
		r = \beta^i(c + d\gamma)
	\]
	for some $i\in\ldb0,\ell-1\rdb$ and $c,d \in M$. Since $\beta$ is a unit, the element $c+d\gamma$ must also be an atom of $S_\beta$. The factorization constructed above rules out the case $c,d \neq 0$. Thus, either $c=0$ or $d=0$, and in either case $r = \beta^j e$ for some $j\in\nn_0$ and some $e \in M$. If $A$ factored into atoms, we would therefore have
	\[
		A = \beta^L c
	\]
	for some $L \in \nn_0$ and $c \in M$. The right-hand side has normal form supported in a single residue class modulo $g$, whereas $A = 1 + \beta^\ell$ has support in the two distinct residue classes $0$ and~$\ell$. This contradicts the uniqueness of \eqref{eq:beta-normal-form}. Thus, $A$ has no factorization into atoms, and $S_\beta$ is not atomic.
\end{proof}

Let us take a look at a concrete example of non-atomic monogenic semidomain

\begin{example}
    Consider the polynomial $m(x) := x^6+5x^2-1 \in \qq[x]$, and let us argue that $m(x)$ is irreducible. First, observe that modulo $3$ it can be written as the product of the irreducible cubics:
    \[
        m(x) = (x^3+x^2+2x+1)(x^3+2x^2+2x+2).
    \]
    On the other hand, we can factor $m(x)$ modulo $7$ as follows: 
    \[
        m(x) = (x^2+1)(x^4+6x^2+6).
    \]
    Then any proper irreducible factor of $m(x)$ in $\zz[x]$ would therefore have degree $3$ by reduction modulo $3$, and such a factor over $3$ has no root. However, the quartic over $7$ has neither a root nor a monic quadratic divisor, as direct division verifies. Hence $m(x)$ is irreducible in $\qq[x]$.
	
    Now let $\beta$ be a root of $m(x)$. Then the relation $1 = \beta^6 + 5\beta^2$ makes $1$ additively reducible and hence makes $S_\beta$ additively antimatter. Here $g=2$, and $\beta^2$ has minimal polynomial $x^3+5x-1$, which has a nonreal conjugate pair because its derivative is positive on $\rr$. Thus Proposition~\ref{prop: accp counterexample} applies: $S_\beta$ is not atomic, and $1+\beta$ has no factorization into atoms.
\end{example}

\bigskip
\section{The Bounded and Finite Factorization Properties}

In this section, we study factorization properties in our semidomain of interest.

\medskip
\subsection{The Finite Factorization Property}

We begin our discussion of the finite factorization property by identifying a natural class of algebraic monogenic semidomains generated by algebraic integers whose members are FFS.

\begin{prop}
	Let $\alpha$ be an algebraic integer with number field $K := \qq(\alpha)$, and let $\mathcal{O}_K$ be the ring of integers of~$K$. If the quotient group $\mathcal{O}_K^\times/S_\alpha^\times$ is finite then $S_\alpha$ is an FFS.
\end{prop}

\begin{proof}
	Set $S := S_\alpha$ and fix a nonzero $b \in S$. If $d \mid_{S_\alpha}b$, then $b=dc$ for some $c \in S$. Hence
	\[
		b\mathcal{O}_K = (d\mathcal{O}_K)(c\mathcal{O}_K).
	\]
	Thus, $d\mathcal{O}_K$ is an ideal divisor of $b\mathcal{O}_K$. Since $\mathcal{O}_K$ is a Dedekind domain, the nonzero ideal $b \mathcal{O}_K$ has only finitely many ideal divisors. Therefore the set
	\[
		\mathcal{D}(b) := \big\{d\mathcal{O}_K : d\in S \text{ and } d \mid_{S_\alpha} b \big\}
	\]
	is finite. Fix $I \in \mathcal{D}(b)$, and choose $d_0 \in S_\alpha$ such that $d_0 \mid_{S_\alpha} b$ and $d_0 \mathcal{O}_K = I$. If $d \in S_\alpha$ is another divisor of $b$ with $d\mathcal{O}_K = I$ then $d\mathcal{O}_K = d_0 \mathcal{O}_K$, so $d = ud_0$ for some $u \in \mathcal{O}_K^\times$. If $u_1S_\alpha^\times = u_2S_\alpha^\times$ then $u_1d_0$ and $u_2d_0$ are associates in~$S_\alpha$. Hence the associate classes of the divisors of $b$ lying above~$I$ are parametrized by a subset of $\mathcal{O}_K^\times/S_\alpha^\times$, which is finite by hypothesis. Since $\mathcal{D}(b)$ is finite, $b$ has only finitely many divisors in $S$ up to associates.

	It remains to verify atomicity. Suppose, to the contrary, that some nonzero nonunit $b_0 \in S_\alpha$ does not factor into atoms. Then $b_0$ is reducible, so we may write $b_0 = b_1c_1$, where $b_1$ and $c_1$ are nonunits and $b_1$ does not factor into atoms. Repeating this argument produces a sequence $(b_n)_{n \ge 0}$ such that $b_n = b_{n+1}c_{n+1}$ for every $n \in \nn_0$ such that every $b_n$ and $c_{n+1}$ is a nonunit. Each $b_n$ divides $b_0$, while $b_{n+1}$ is not associate to $b_n$ because the quotient $c_{n+1}$ is not a unit. Consequently, the elements $b_n$ are pairwise nonassociate divisors of $b_0$, contradicting the divisor-finiteness proved above. Thus, $S_\alpha$ is atomic. The divisor-finiteness characterization of finite factorization monoids~\cite{fHK92} now yields that $S_\alpha$ is an FFS.
\end{proof}

Every order in a number field is an FFD; this is already contained in \cite[Example~1]{AM96}. We include a direct proof for completeness.

\begin{theorem}\label{thm:orders are FFDs}
	Every order in a number field is an FFD.
\end{theorem}

\begin{proof}
    Let $R$ be an order in a number field $K$. For each nonzero $a\in R$, multiplication by $a$ is an injective endomorphism of the finite-rank free abelian group $R$. Its determinant is the nonzero integer $N_{K/\qq}(a)$, so
    \[
        [R:aR]=|N_{K/\qq}(a)|<\infty.
    \]
    If $d\mid_R a$, then $aR\subseteq dR\subseteq R$. The finite additive group $R/aR$ has only finitely many subgroups, so there are only finitely many possibilities for $dR$. Two nonzero generators of the same principal ideal are associates. Thus every nonzero element has finitely many divisors up to associates, and $R$ is an FFD.
\end{proof}

Given the relevance of ring theory, it is important in the scope of this paper to understand the factorization behavior of algebraic monogenic semidomains that are rings, namely, rings of the form $\zz[\alpha]$ for $alpha \in \aaa$. If $\alpha$ is an algebraic integer of degree $d$ then $\zz[\alpha]$ is a free $\zz$-module with basis $\{1,\alpha,\dots,\alpha^{d-1} \}$. Moreover, $\zz[\alpha]$ is contained in the ring of integers of $\qq(\alpha)$ and has full rank~$d$. Therefore $\zz[\alpha]$ is always an order in $\qq(\alpha)$, and so there is no algebraic integer $\alpha$ for which $\zz[\alpha]$ is not an order. Here the integrality assumption is essential as for $\alpha = 1/2$, the ring $\zz[\alpha]=\zz[1/2]$ is not an order in $\qq$: it is not contained in the ring of integers $\zz$ and is not finitely generated as a $\zz$-module. Orders of the form $\zz[\alpha]$ for some algebraic integer $\alpha$ are called {monogenic orders}. While every order in a quadratic number field $\mathbb{Q}(\sqrt{d})$ is monogenic, this behavior does not hold in general for number fields of degree $n \ge 3$. Indeed, even ring of integers fail to be monogenic in general. The following classic counterexample is due to Dedekind.

\begin{example}
	Consider the cubic field $K = \mathbb{Q}(\theta)$, where $\theta$ is a root of $f(x) = x^3 - x^2 - 2x - 8$. The maximal order $\mathcal{O}_K$ has a $\mathbb{Z}$-module basis $\{ 1, \theta, (\theta^2 + \theta)/2 \}$. Dedekind proved that the ideal $(2)$ generated by $2$ in $\OK$ factors into three distinct primes ideals. If $\mathcal{O}_K$ were equal to $\mathbb{Z}[\alpha]$ for some $\alpha$ then, by Kummer-Dedekind theorem, the factorization of $(2)$ in $\mathcal{O}_K$ would correspond to the factorization of the minimal polynomial $\bar{g}(x)$ of $\alpha$ in $\ff_2[x]$. However, a cubic polynomial $\bar{g}(x) \in \mathbb{F}_2[x]$ can have at most two roots in $\mathbb{F}_2$ (since $\mathbb{F}_2$ has only two elements, $0$ and $1$). Thus, $\bar{g}(x)$ can have at most two distinct linear factors in $\ff_2[x]$, meaning that $(2)$ could split into at most two prime factors of residue degree 1 in $\mathbb{Z}[\alpha]$. Because $(2)$ splits into three degree-1 prime factors in $\mathcal{O}_K$, no such generator $\alpha$ can exist. Hence the ring of integers $\OK$ is not monogenic.
\end{example}

Let us prove now that every simple ring extension of the prototypical integral domain $\zz$ is an FFD.

\begin{prop}\label{prop:Z-alpha-is-FFD}
	For every algebraic number $\alpha$, the integral domain $\zz[\alpha]$ is an FFD.
\end{prop}

\begin{proof}
	Fix $\alpha \in \aaa$, and set $R := \zz[\alpha]$. Let $g(x) \in \zz[x]$ be the primitive integer polynomial obtained by multiplying the minimal polynomial of $\alpha$ by a suitable positive rational number. The evaluation homomorphism $\zz[x] \to R$ has kernel $g(x) \zz[x]$, whence
	\[
		R \cong \zz[x]/(g).
	\]
	In particular, $R$ is a Noetherian domain and, as a result, atomic. In light of Anderson and Mullins~\cite[Theorem~6]{AM96}, we are done once we prove that the quotient group $T^\times/R^\times$ is finite for every overring $T$ of $R$ that is finitely generated as an $R$-module. Fix one of such overrings, say $T$, and then write
	\[
		T = t_1R + \cdots + t_sR
	\] 
	for some $t_1, \dots, t_s \in T$. Since $T \subseteq \qq(\alpha) = \qq[\alpha]$, there exists $m \in \nn$ such that $mt_i \in R$ for every $i \in \ldb 1,s \rdb$. Therefore $mT \subseteq R$. Let us now argue the following claim.
	\medskip

	\noindent \textsc{Claim.} $T/mT$ is finite.
	\smallskip

	\noindent \textsc{Proof of Claim.} We first show that the quotient $R/mR$ is finite and from there we will obtain that $T/mT$ is also finite. For every $p \in \pp$, the primitivity of $g(x)$ ensures that its image $\overline{g}(x)$ in $\ff_p[x]$ is nonzero, and so
	\[
		R/pR \cong \ff_p[x]/(\overline{g})
	\]
	is finite. For each $e \in \nn$, consider now the homomorphism $R/pR \to p^{e-1}R/p^eR$ defined via the assignment $r+pR \mapsto p^{e-1}r+p^eR$ for all $r \in R$, which is clearly surjective. In addition, the following short exact sequence
	\[
		0 \longrightarrow p^{e-1}R/p^eR \longrightarrow R/p^eR \longrightarrow R/p^{e-1}R\longrightarrow0,
	\]
	shows inductively that $R/p^eR$ is finite for every $e \in \nn$. The Chinese remainder theorem now yields that $R/mR$ is finite. Since $T$ is generated by $t_1,\dots,t_s$ as an $R$-module, the natural map $(R/mR)^s\longrightarrow T/mT$ defined as
	\[
		(r_1 + mR,\dots,r_s + mR) \longmapsto \sum_{i=1}^s t_ir_i+ mT,
	\]
	is a surjective homomorphism whose domain is finite, so the quotient group $T/mT$ is also finite, and the claim is established.
	\smallskip

	Now consider the reduction homomorphism $\rho \colon T^\times \longrightarrow(T/mT)^\times$, and let us show that $\ker\rho \subseteq R^\times$. Indeed, if $u \in \ker \rho$ then $u \in 1 + mT$ or, equivalently, $u-1 \in mT \subseteq R$, whence $u \in R$. Moreover, $u^{-1} \in T$ is integral over $R$ because $T$ is a finitely generated $R$-module. Therefore we can write
	\[
		(u^{-1})^n + r_{n-1}(u^{-1})^{n-1} + \dots + r_0 = 0
	\]
	for some $r_0, \dots, r_{n-1} \in R$ and, after multiplying both sides of the equality by $u^{n-1}$, we obtain that $u^{-1} \in R$. Hence $u \in R^\times$, proving that $\ker \rho \subseteq R^\times$.

	Finally, choose a set of representatives for the cosets of $R^\times$ in $T^\times$. If two representatives have the same image under $\rho$, then their quotient belongs to $\ker\rho\subseteq R^\times$, so they represent the same coset. Hence this set of representatives injects into the finite group $(T/mT)^\times$. Then it follows that quotient group $T^\times/R^\times$ is finite. Thus, in light of Anderson and Mullins' criterion, we conclude that $R = \zz[\alpha]$ is an FFD.
\end{proof}

Now we can identify two choices of the generator $\alpha \in \cc$ such that the monogenic semidomain $S_\alpha$ has the finite factorization property.

\begin{prop} \label{prop:S_alpha is an FFD if no + conj & an FFS if conj at least one}
	Let $\alpha$ b a nonzero algebraic number. Then the following statements hold.
	\begin{enumerate}
		\item If $\alpha$ has no positive conjugates then $S_\alpha$ is an FFD.
		\smallskip

		\item If $\alpha$ has a positive conjugate that is at least $1$ then $S_\alpha$ is an FFS.
	\end{enumerate}
\end{prop}

\begin{proof} 
	(1) Assume that $\alpha$ does not have a positive conjugate. Then $S_\alpha$ is a ring, which means that $S_\alpha = \zz[\alpha]$. Thus, it follows from Proposition~\ref{prop:Z-alpha-is-FFD} that $S_\alpha$ is an FFD.
	\smallskip

	(2) After replacing $\alpha$ by its largest positive conjugate, we can assume that $\alpha > 1$, and we can do such a replacement because conjugate algebraic numbers generate isomorphic monogenic semidomains. As $\alpha > 1$ we see that $S^*_\alpha$ is isomorphic to the increasing positive monoid $\ln S^*_\alpha$ of the ordered field $\rr$, whence it follows from \cite[Theorem~5.6]{fG19} that $S_\alpha$ is an FFS. 
\end{proof}

\medskip
\subsection{The Bounded Factorization Property}

Next, we prove that in the class of algebraic monogenic semidomain, the bounded factorization property is equivalent to the ACCP. For each positive algebraic number $\alpha$, we set
\[
	D_\alpha:=\{\sigma:\qq(\alpha)\hookrightarrow\cc:
	\sigma\text{ fixes }\qq,\ |\sigma(\alpha)|\le\alpha\},
	\qquad \|z\|_\alpha:=\max_{\sigma\in D_\alpha}|\sigma(z)|
\]
for all $z \in \qq(\alpha)$. A \emph{Perron number} is a real algebraic integer $\lambda > 1$ whose modulus strictly exceeds the modulus of every other conjugate.

\begin{lemma}\label{lem:neighborhoodmembership}    
    Let $\alpha$ be an algebraic number with $\alpha \in (0,1)$ and having no larger positive conjugate. Then there exist $h \in S_\alpha^*$ and $\epsilon > 0$ such that $hz\in S_\alpha$ whenever $z \in \Za$ and $\|z-1\|_\alpha < \epsilon$.
\end{lemma}

\begin{proof}
    We first claim that there exists $M \in \mathbb N$ so that if $w\in \Z[\alpha]$ and $\|w \|_\alpha <1$, then there exists $f(x)\in \Z[x]$ with every coefficient at least $-M$ and $f(\alpha)=w$. Fix such a $w$ and choose $g(x)\in \Z[x]$ so that $g(\alpha)=w$, and let $\Sigma$ be the set of conjugates of $\alpha$. Apply \cite[Corollary~2.4]{NW98} to the function on $\Sigma$ which is 0 at the conjugates of $\alpha$ with magnitude at most $\alpha$, and is equal to $g$ for the other conjugates. In the notation of that result, $\kappa(\Sigma)=\alpha$ and, for $V=\{z:|z|>\alpha \}$, the set $\cc\setminus\Sigma_\kappa = V \setminus \Sigma$ is connected and unbounded, while $V \cap \Sigma$ is finite. The corollary therefore gives $P(x) \in \R_{\ge 0}[x]$ so that, for every $z\in \Sigma$, 
    \[|P(z)|<1 \quad \text{if }|z| \le \alpha, \quad \quad |P(z)-g(z)|<1 \quad \text{if } |z|>\alpha.\] Let $E(x)\in \rr[x]$ be the remainder of dividing $g-P$ by $m_\alpha(x)$. Equivalently, it is the polynomial of degree less than $\deg m_\alpha(x)$ that interpolates the values of $g-P$ on $\Sigma$. Since $\Sigma$ is fixed, and the magnitude of $g-P$ evaluated at $\Sigma$ is bounded by $2$ (recall that $\|w\|_\alpha<1)$, the coefficients of $E$ are bounded independently of $w$. Now $P+E-g$ vanishes on $\Sigma$, so write $P+E-g=\omega_\alpha Q$ with $Q\in \rr[x]$. Round the coefficients of $Q$ to obtain $\tilde{Q}(x)\in \Z[x]$, and let $f=g+\tilde{Q}\omega_\alpha$. Then $f(\alpha)=w$ and $f=P+E+\omega_\alpha (\tilde{Q}-Q)$. The coefficients of the last two terms are uniformly bounded, and $P$ has nonnegative coefficients, so the coefficients of $f$ are bounded below as claimed.

    Recall that every nonnegative zero of $\omega_\alpha$ is smaller than $1$, so $\omega_\alpha$ has no roots in $[1,\infty)$. Since the function $(x-1)/\omega_\alpha(x)^2$ is continuous and tends to $0$ at infinity, it is bounded in $[1,\infty)$. Hence there is some $c \in \nn$ so that $1-x+c\omega_\alpha(x)^2$ is positive on $\rr_{\ge 0}$. P\'olya's theorem gives $n \in \nn_0$ so that
	\[
		(1+x)^n(1 - x + c\omega_\alpha(x)^2)
	\]
	has strictly positive coefficients. Now choose $q \in \nn$ so that for $A(x) \coloneq q(1+x)^n$ and $B(x)\coloneq A(x)(1-x+c\omega_\alpha (x)^2)$, we have that every coefficient of $B$ is at least $M$. By iteratively applying $A(\alpha)=B(\alpha)+\alpha A(\alpha)$, we get that for each $N \in \nn$ the polynomial\[C_N(x) \coloneq B(x)(1+x+\cdots +x^N)+x^{N+1}A(x)\] gives $A(\alpha)$ when evaluated at $\alpha$. For every integer $j$ with $0\le j\le N+\deg B$, the coefficient $[x^j]C_N$ is at least $M$: some term of $B(x)(1+x+\cdots+x^N)$ contributes a coefficient of $B$ at that degree. Moreover, $\deg B=\deg A+2\deg\omega_\alpha>\deg A+1$, so $\deg C_N=N+\deg B$.

    Now let $h = A(\alpha)$, and let $\varepsilon = \|h\|_\alpha^{-1}$. If $\|t-1\|_\alpha < \varepsilon$ then $\| h(t-1) \|_\alpha < 1$. Take the representative of $h(t-1)$ supplied by the earlier claim, and pick $N$ so that its degree is at most $N + \deg B$. Adding this to $C_N(x)$ gives a polynomial in $\nn_0[x]$ representing $h(t-1)+h = ht$, so $ht \in S_\alpha$.
\end{proof}

\begin{lemma}\label{lem:divisorbound}
    Let $\alpha$ be an algebraic number with $\alpha \in (0,1)$ and no larger positive conjugate. For each $x \in S_\alpha^*$, there exists $C \ge 1$ such that each divisor of $x$ in $S_\alpha$ is associate in $S_\alpha$ to a divisor $y$ satisfying
    \[
        C^{-1} \le |\sigma(y)| \le C
    \]
	for all $\sigma \in D_\alpha$.
\end{lemma}

\begin{proof}
    If $x = yz$ for some $y,z \in S_\alpha^*$ then nonnegative coefficients give $|\sigma(y)| \le y$ and $|\sigma(z)| \le z$ for every $\sigma \in D_\alpha$. Consequently,
    \[
        y \frac{|\sigma(x)|}{x} \le |\sigma(y)| \le y.
    \]
    It suffices to bound the positive real value of a suitable associate of $y$ above and away from zero.

    If the group $S_\alpha^\times$ is nontrivial then $\alpha$ is a unit. Multiplying $y$ by an integral power of $\alpha$ gives an associate divisor in $[1,\alpha^{-1})$, as required. 

    Assume therefore that $S_\alpha^\times$ is the trivial group and so that $S_\alpha$ is reduced. First, consider the case where $\alpha \notin \zz[\alpha]^\times$. Some nonzero prime-ideal valuation $v$ of $\qq(\alpha)$ satisfies $v(\alpha) > 0$: otherwise $\alpha^{-1}$ would be integral over $\zz$, and its monic equation would imply $\alpha^{-1} \in \zz[\alpha]$. At this place, $v(s)\ge 0$ for every nonzero $s \in S_\alpha$. Choose $f,g \in\nn_0[x]$ with $y = f(\alpha)$ and $z = g(\alpha)$, and then set $j := \min \supp f$ and $k := \min \supp g$. Then
    \[
        v(x) = v(y) + v(z) \ge (j+k)v(\alpha).
    \]
    Hence $j,k \le J := \lfloor v(x)/v(\alpha)\rfloor$, and
    \[
        y \ge \alpha^j \ge \alpha^J, \qquad z \ge \alpha^J, \qquad y \le x \alpha^{-J}.
    \]

    Finally, suppose the group $S_\alpha^\times$ is trivial but $\alpha \in \zz[\alpha]^\times$. We claim that some conjugate $\rho$ of $\alpha$ satisfies $|\rho| < \alpha$. Otherwise all conjugates of $\alpha^{-1}$ have modulus at most $\alpha^{-1}$. Since $\alpha F(\alpha) = 1$ for some $F \in \zz[x]$, Gauss's lemma ensures that $|\omega_\alpha(0)| = 1$, so $\alpha^{-1}$ is an algebraic integer. Boyd's theorem~\cite{dB94} implies that its simplified minimal polynomial has a Perron root. There is no second positive conjugate of $\alpha$: maximality excludes larger ones, and the assumed modulus bound excludes smaller ones. Corollary~4.6 of~\cite{CGLY26} therefore makes $S_\alpha$ additively antimatter, contradicting Proposition~\ref{prop:nontrivial unit group additive antimatter}. This proves the claim.
    \smallskip

    Fix an embedding $\tau$ with $\tau(\alpha)=\rho$, and again write $y=f(\alpha)$, $z=g(\alpha)$, with least exponents $j,k$. Put $r=|\rho|/\alpha\in(0,1)$. The polynomial $F=fg$ has least exponent $j+k$, and
    \[
        0 < |\tau(x)|=|F(\rho)|\le F(|\rho|)
        \le r^{j+k}F(\alpha)=r^{j+k}x.
    \]
    Thus $j+k\le\log(|\tau(x)|/x)/\log r$. Taking the integer part of this fixed nonnegative bound gives $J$ with $j,k\le J$, and the preceding inequalities $\alpha^J\le y\le x\alpha^{-J}$ apply again. The initial comparison with $|\sigma(y)|$ finishes the proof.
\end{proof}

\begin{theorem}\label{thm:bfiffaccp}
    The following conditions are equivalent for any $\alpha \in \aaa$.
	\begin{enumerate}
		\item[(a)] $S_\alpha$ is a BFS.
		\smallskip

		\item[(b)] $S_\alpha$ satisfies ACCP.
	\end{enumerate}
\end{theorem}

\begin{proof}
	(a) $\Rightarrow$ (b): This is true in general.
	\smallskip

	(b) $\Rightarrow$ (a): Assume that $S_\alpha$ satisfies the ACCP. If $\alpha = 0$ or a conjugate of $\alpha$ belongs to $\nn_0$ then $S_\alpha = \nn_0$. If $\alpha$ has no nonnegative real conjugate then $S_\alpha = \zz[\alpha]$ and Proposition~\ref{prop:Z-alpha-is-FFD} applies. Otherwise, replace $\alpha$ by its largest positive conjugate. If $\alpha > 1$ then $S^*_\alpha$ is isomorphic to the increasing positive monoid $\ln S_\alpha^*$ of the ordered field $\rr$, and so $S_\alpha$ is an FFS. Thus, we assume for the rest of the proof that $0 < \alpha < 1$. Now consider the set
	\[
		T := (\zz[\alpha]^\times \cap S_\alpha) \setminus S_\alpha^\times.
	\]
	
	We claim that there exists $\varepsilon > 0$ such that $\|t-1\|_\alpha \ge \varepsilon$ for all $t \in T$. Suppose otherwise, and choose $h$ and $\varepsilon_0$ from Lemma~\ref{lem:neighborhoodmembership}. Choose $\delta \in (0,1)$ with $e^\delta - 1 < \varepsilon_0$. Continuity of inversion at $1$ allows us to choose for each $n \in \nn$ a value $t_n \in T$ so that
    \[
        \|t_n^{-1}-1\|_\alpha < \frac{\delta}{2^n}.
    \]
    Set $p_0 := 1$ and $p_n := t_1 \cdots t_n$ for every $n \in \nn$. We show that all inverse partial products remain in the same neighborhood of~$1$. Fix $\sigma \in D_\alpha$ and put $z_j := \sigma(t_j^{-1}) - 1$, so that $|z_j|<\delta/2^j$. Since $\sigma$ is multiplicative, $\sigma(p_n^{-1})=\prod_{j=1}^n(1+z_j)$. Expanding this product and subtracting its constant term, then applying the triangle inequality to all remaining terms, gives
    \[
        \bigg|\prod_{j=1}^n(1+z_j)-1\bigg|
        \le \prod_{j=1}^n(1+|z_j|)-1
        \le \prod_{j=1}^n\left(1+\frac{\delta}{2^j}\right)-1.
    \]
    Now use $1+u\le e^u$ for $u\ge0$ and the geometric-sum identity $\sum_{j=1}^n2^{-j}=1-2^{-n}<1$ to obtain
    \[
        \prod_{j=1}^n\left(1+\frac{\delta}{2^j}\right)
        \le \exp\bigg(\delta\sum_{j=1}^n2^{-j}\bigg)
        =e^{\delta(1-2^{-n})}<e^\delta.
    \]
    Thus, $|\sigma(p_n^{-1})-1|<e^\delta-1<\varepsilon_0$, uniformly in $n$ and $\sigma$. Taking the maximum over $\sigma\in D_\alpha$ yields $\|p_n^{-1}-1\|_\alpha<\varepsilon_0$ for every $n\in\nn_0$, with $n=0$ immediate. Moreover, each $p_n^{-1}$ belongs to $\zz[\alpha]$, because every $t_j$ is a unit of this ring. We can therefore apply Lemma~\ref{lem:neighborhoodmembership} to every inverse partial product using the same element~$h$.
    Hence $b_i = hp_i^{-1} \in S_\alpha$ for every $i$, and $b_i=t_{i+1}b_{i+1}$ gives a strictly ascending chain $b_iS_\alpha\subsetneq b_{i+1} S_\alpha$, because $t_{i+1}$ is a nonunit. This contradicts the ACCP and proves the claim.
	\smallskip

	We are now in a position to argue that $S_\alpha$ is a BFS. For this, fix $x \in S_\alpha^*$. As $\zz[\alpha]$ is a BFD, there is a bound on the number of factors that are nonunits of $\zz[\alpha]$ in any factorization of $x$ into nonunit factors of $S_\alpha$. List the remaining factors as $t_1, \dots, t_m \in T$. Now, for each $j \in \ldb 1,m \rdb$, set $p_j := t_1\cdots t_j$ and observe that, as $p_j$ divides $x$ in $S_\alpha$, Lemma~\ref{lem:divisorbound} ensures the existence of $u_j \in S_\alpha^\times$ such that $q_j = p_ju_j$ and
    \[
        C^{-1} \le |\sigma(q_j)| \le C
    \]
    for all $\sigma \in D_\alpha$, where $C$ depends only on~$x$. The corresponding vectors belong to the compact product of complex annuli $\{(z_\sigma):C^{-1}\le |z_\sigma| \le C \}$. Cover this compact set by $N$ sets of diameter less than $\varepsilon/C$ in the maximum metric. If $m > N$ then two vectors, indexed by $j>k$, lie in the same set, so
    \[
        \left\|q_j q_k^{-1}-1\right\|_\alpha < \varepsilon.
    \]
    However,
    \[
        q_jq_k^{-1}=t_{k+1}\cdots t_j\,u_ju_k^{-1}\in T,
    \]
    contradicting the claim. Thus, $m\le N$, which bounds all factorization lengths. The ACCP implies atomicity, so $S_\alpha$ is a BFS.
\end{proof}

In light of Proposition~\ref{prop:S_alpha is an FFD if no + conj & an FFS if conj at least one}, if the generator $\alpha$ of a semidomain $S_\alpha$ has no positive conjugate or a positive conjugate that is at least $1$ then $S_\alpha$ has the finite factorization property. The following theorem provides a complete the characterization of the finite factorization property by addressing the generators $\alpha$ having all their positive conjugates inside the open interval $(0,1)$.

\begin{theorem} \label{thm:ffcharacterization}
    Let $\alpha \in (0,1)$ be an algebraic number with no positive conjugate larger than itself, let $m_\alpha(x)$ be the minimal polynomial of $\alpha$, and set
    \[
        g := \gcd\,\emph{supp}\,m_\alpha(x),\qquad
        \lambda:=\alpha^{-g},\qquad K:=\qq(\lambda).
    \]
    Let $e$ be the number of conjugates $\rho$ of $\alpha$ with $|\rho|>\alpha$, counting each nonreal conjugate pair once, and let $s$ be the number of denominator primes of $\alpha$ in $\mathcal O_{\qq(\alpha)}$. Then the following statements hold.
    \begin{enumerate}
    	\item[(1)] If $\lambda$ is not a Perron number 
		then $S_\alpha$ is an FFS if and only if $e+s \le 1$.
		\smallskip
    
		\item[(2)] If $\lambda$ is a Perron number then $S_\alpha$ is an FFS precisely in the following cases: 
		\begin{itemize}
			\item $g=1$;
			\smallskip

			\item  $g=2$, the field $K$ is totally real, and the ideals $\lambda \mathcal O_{\qq(\lambda)}$ and $\lambda \mathcal O_{\qq(\alpha)}$ have the same number of distinct prime divisors;
			\smallskip

			\item $g \ge 3$, the field $K$ equals $\Q$, and the ideals $\lambda \mathcal O_{\qq(\lambda)}$ and $\lambda \mathcal O_{\qq(\alpha)}$ have the same number of distinct prime divisors.
		\end{itemize} 
    \end{enumerate}
\end{theorem}

\begin{proof}
    Let $L$ and $\eta$ denote the field $\qq(\alpha)$ and the element $\alpha^g :=\lambda^{-1}$, so that $K=\qq(\lambda)=\qq(\eta)$. It follows from the definition of~$g$ that $m_\alpha(x) = f(x^g)$ for some simple monic polynomial $f \in \qq[x]$. The irreducibility of $m_\alpha(x)$ implies the irreducibility of $f(x)$, whence the equality $f(\eta)=0$ ensures that $f(x) = m_\eta(x)$. After comparing degrees, we obtain that $[L:K] = g$, so $x^g - \eta$ is the minimal polynomial of $\alpha$ over $K$. In particular, $1, \alpha, \dots, \alpha^{g-1}$ form a $K$-basis for~$L$.

    Let $S$ be the set of denominator primes of $\alpha$ in $\mathcal O_L$, and let $\Sigma_L$ consist of $S$ and the Archimedean places of $L$. For $a\in L^\times$, write $\log a:=(\log|a|_v)_{v\in\Sigma_L}$, using the product-formula normalization, and set
    \[
        H_L:=\Big\{(x_v)_{v\in\Sigma_L}\in\rr^{\Sigma_L}:
                    \sum_{v\in\Sigma_L}x_v=0\Big\}.
    \]
    The product formula puts the logarithmic image of every $S$-unit in $H_L$ as its absolute values at finite places outside $S$ are all~$1$.

    By the description of overrings of a Dedekind domain~\cite[Proposition~1.1]{lC65},
    \[
        \mathcal O_L[\alpha]
        =\bigcap_{\substack{\mathfrak q\text{ a nonzero prime of }\mathcal O_L\\
                   \mathcal O_L[\alpha]\subseteq(\mathcal O_L)_{\mathfrak q}}}
                   (\mathcal O_L)_{\mathfrak q}
        =\bigcap_{\mathfrak q\notin S}(\mathcal O_L)_{\mathfrak q}
        =\mathcal O_{L,S}.
    \]
    Here the second intersection also runs over nonzero prime ideals of $\mathcal O_L$. The middle equality follows because $\mathcal O_L[\alpha]\subseteq(\mathcal O_L)_{\mathfrak q}$ if and only if $v_{\mathfrak q}(\alpha) \ge 0$.

    If $\omega_1,\ldots,\omega_d$ form a $\zz$-basis of $\mathcal O_L$ then $\mathcal O_L[\alpha]=\omega_1\Za+\cdots+\omega_d\Za$. Thus, $\mathcal O_L[\alpha]$ is a finitely generated $\Za$-module, and the proof of Proposition~\ref{prop:Z-alpha-is-FFD} yields the following inequality:
    \[
        [\mathcal O_{L,S}^\times : \Za^\times] < \infty.
    \]
    The subgroup $\Za^\times\cap\rr_{>0}$ has index~$2$ in $\Za^\times$. Consequently, it follows from Theorem~\ref{thm:s unit theorem} that
    \[
        \Lambda := \log(\Za^\times\cap\rr_{>0})
    \]
    is a full lattice in $H_L$. Moreover, the logarithmic map is injective on positive elements of $L^\times$: its coordinate at the identity embedding is the ordinary real logarithm.

    Let $\mathcal R_\alpha$ be the set of Archimedean places of $L$ represented by embeddings in $D_\alpha$, and consider the map $\pi \colon H_L\rightarrow \rr^{\mathcal R_\alpha}$ defined as
    \[
        \pi \colon (x_v)_{v\in\Sigma_L} \mapsto (x_v)_{v\in\mathcal R_\alpha}.
    \]
    Put $W:=\operatorname{Span}_{\rr}\log(S_\alpha^\times)$ and $\Gamma:=\Lambda\cap W$. Since $\log(S_\alpha^\times)$ is a subgroup of the lattice $\Lambda$, it is discrete and is a full lattice in its real span~$W$. The same holds for $\Gamma$, so
    \[
        [\Gamma:\log(S_\alpha^\times)] < \infty.
    \]
	
    \noindent \textsc{Claim.} $S_\alpha$ is an FFS if and only if $\ker\pi\subseteq W$.
    \smallskip

    \noindent \textsc{Proof of Claim.} Suppose first that $\ker\pi\subseteq W$, and choose a linear complement $V$ of $W$ in $H_L$. Then $\pi(H_L)=\pi(V)\oplus\pi(W)$ and $\pi|_V \colon V \rightarrow \pi(V)$ is an isomorphism. Indeed, if $\pi(v) = \pi(w)$ with $v\in V$ and $w \in W$ then $v-w \in \ker \pi \subseteq W$, whence $v \in V \cap W = \{0\}$. Taking $w=0$ also proves injectivity of $\pi|_V$.

    Fix $x \in S_\alpha$. Proposition~\ref{prop:Z-alpha-is-FFD} and the divisor-finiteness characterization~\cite[Corollary~2]{fHK92} give only finitely many $\Za^\times$-associate classes of divisors of $x$ in $\Za$. Consider one that contains a divisor $d$ of $x$ in $S_\alpha$, and write its other divisors of $x$ in $S_\alpha$ as $da$, where $a\in\Za^\times\cap\rr_{>0}$. Lemma~\ref{lem:divisorbound} supplies, for each such $a$, a unit $u\in S_\alpha^\times$ for which $|\sigma(dau)|$ is bounded above and away from zero, uniformly over $a$ and $\sigma\in D_\alpha$. Since $d$ is fixed, $\pi(\log(au))$ lies in a fixed bounded set.

    Write $\log a=v+w$ with $v\in V$ and $w\in W$. Since $\log u\in W$, the $\pi(V)$-component of $\pi(\log(au))$ is $\pi(v)$. Linear projections and the inverse of $\pi|_V$ are continuous, so $v$ ranges over a bounded subset of~$V$. Choose a bounded fundamental parallelepiped $P$ for the lattice $\Gamma$ in $W$. For each $w$, choose $\gamma\in\Gamma$ such that $w-\gamma\in P$. Then
    \[
        \log a-\gamma=v+(w-\gamma)\in\Lambda
    \]
    belongs to a fixed bounded set, which contains only finitely many points of~$\Lambda$. Hence the possible $\log a$ occupy finitely many cosets modulo $\Gamma$, and the finite index $[\Gamma:\log(S_\alpha^\times)]$ gives finitely many cosets modulo $\log(S_\alpha^\times)$. Injectivity on positive units now gives finitely many $S_\alpha^\times$-associate classes of the divisors $da$. Applying this argument to the finitely many ring-associate classes shows that $S_\alpha$ is an FFS.

    Conversely, suppose that $\ker\pi\nsubseteq W$, and choose $\xi\in\ker\pi\setminus W$. A bounded fundamental parallelepiped for $\Lambda$ allows us to choose $a_n\in\Za^\times\cap\rr_{>0}$ such that $\log a_n-n\xi$ remains bounded. Since $\pi(\xi)=0$, the values $|\sigma(a_n)|$ are bounded above and away from zero for all $\sigma\in D_\alpha$. The vectors $(\sigma(a_n))_{\sigma\in D_\alpha}$ therefore lie in a compact product of closed annuli. After passing to a subsequence, we may assume that each $\sigma(a_{n_j})$ converges to a nonzero limit; by thinning this subsequence, we may also arrange that $n_{j+1}-n_j\to\infty$.

    Let $q_j:=a_{n_{j+1}}/a_{n_j}$. Both $q_j$ and its inverse belong to $\Za$, since the numerator and denominator are units of this ring. For every embedding $\sigma\in D_\alpha$, the numerator and denominator of $\sigma(q_j)$ approach the same nonzero limit, so both $\sigma(q_j)$ and $\sigma(q_j^{-1})$ approach~$1$. Moreover, the difference between $\log q_j$ and $(n_{j+1}-n_j)\xi$ remains bounded, since it is the difference of two terms of the bounded sequence $\log a_n-n\xi$. Choose a real linear functional $\ell$ on $H_L$ whose restriction to $W$ is the zero linear functional and $\ell(\xi) = 1$. Then $\ell(\log q_j) \to \infty$, and so a further subsequence has pairwise distinct logarithms modulo~$W$, and so pairwise distinct classes modulo $S_\alpha^\times$. Since $D_\alpha$ is finite, both $\|q_j-1\|_\alpha$ and $\|q_j^{-1}-1\|_\alpha$ tend to zero. Lemma~\ref{lem:neighborhoodmembership} therefore gives a single $h\in S_\alpha^*$ such that $hq_j, hq_j^{-1} \in S_\alpha$ for all sufficiently large~$j$. The factorizations
    \[
        h^2 = (hq_j)(hq_j^{-1})
    \]
    exhibit infinitely many pairwise nonassociate divisors of $h^2$. As a consequence, we conclude that $S_\alpha$ is not an FFS, which establishes the claim. 
    \smallskip

    We next determine when $\ker\pi\subseteq W$. By Propositions~\ref{prop:nontrivial units when N_0[alpha] is not an ID} and~\ref{prop:nontrivial unit group additive antimatter}, and \cite[Proposition~3.4 and Theorem~5.6]{CGLY26},
    \[
        S_\alpha^\times\ne\{1\}
        \quad\Longleftrightarrow\quad
        \lambda\text{ is a Perron number}.
    \]
    To check the positive-conjugate condition in the cited theorem, suppose that $\lambda$ is Perron and has another positive conjugate $\lambda'$. Then $0<\lambda'<\lambda$, so the positive root $(\lambda')^{-1/g}$ of $m_\alpha$ would be larger than $\alpha$, contrary to our hypothesis. Thus $\lambda$ has no other positive conjugate. Conversely, the cited equivalences make $S_\eta$ a simple antimatter monoid whenever $S_\alpha$ is additively antimatter, and hence make $\eta^{-1}=\lambda$ a Perron number.

    If $\lambda$ is not Perron, then $S_\alpha^\times=\{1\}$ and $W=\{0\}$. The coordinates omitted by $\pi$ are exactly the $e$ Archimedean places represented by conjugates of modulus greater than $\alpha$ and the $s$ denominator primes. Therefore
    \[
        \dim\ker\pi=\max\{e+s-1,0\}.
    \]
    Indeed, the omitted coordinates satisfy just the relation that their sum is zero; if there are no omitted coordinates, the kernel is already zero. The claim now shows that $S_\alpha$ is an FFS if and only if $e+s\le1$, proving~(1).

    Suppose henceforth that $\lambda$ is Perron. The same positive-conjugate argument shows that $\alpha$ has no other positive conjugate. The preceding equivalences give $\alpha\in S_\alpha^\times$ and make $S_\eta$ additively antimatter. Since $m_\eta$ is simple, \cite[Theorem~5.6]{CGLY26} makes $(S_\eta,+)$ a valuation monoid. Its difference group is $\zz[\eta]$; as $S_\eta\subseteq\rr_{\ge0}$, this implies
    \[
        S_\eta=\zz[\eta]\cap\rr_{\ge0}.
    \]
    Grouping powers of $\alpha$ by their residues modulo $g$ and using the $K$-basis established above gives the unique additive normal form
    \[
        S_\alpha=\bigoplus_{i=0}^{g-1}S_\eta\alpha^i.
    \]
    These normal forms also give
    \[
        S_\alpha^\times
        =\{u\alpha^i:u\in S_\eta^\times,\ 0\le i<g\},
        \qquad [S_\alpha^\times:S_\eta^\times]=g.
    \]
    Indeed, $\eta$ is a unit of $S_\eta$ by Proposition~\ref{prop:nontrivial unit group additive antimatter}. In the product of the normal forms of a unit and its inverse, positivity prevents cancellation. Every pair of nonzero coefficients, therefore, contributes to residue class~$0$. Fixing a nonzero coefficient in either factor forces the other factor's support to be a singleton, so both supports are singletons. The resulting coefficient equation has the form $uv\eta^k=1$, which forces $u,v\in S_\eta^\times$. Conversely, every displayed $u\alpha^i$ is a unit. Finally, if two cosets with $0\le i<j<g$ were equal, then $\alpha^{j-i}\in K$, contradicting the linear independence of $1,\alpha,\ldots,\alpha^{g-1}$ over~$K$. 

    Let $\Sigma_K$ consist of the Archimedean places of $K$ and the nonzero prime ideals dividing $\lambda\mathcal O_K$. For a nonzero prime ideal $\mathfrak q$ of $\mathcal O_L$ above a prime ideal $\mathfrak p$ of $\mathcal O_K$, normalized valuations satisfy $v_{\mathfrak q}|_{K^\times}=e(\mathfrak q/\mathfrak p)v_{\mathfrak p}$, where $e(\mathfrak q/\mathfrak p)>0$ is the ramification index. Since $\alpha^g=\lambda^{-1}$,
    \[
        g v_{\mathfrak q}(\alpha)
        =-e(\mathfrak q/\mathfrak p)v_{\mathfrak p}(\lambda).
    \]
    Hence $v_{\mathfrak q}(\alpha)<0$ if and only if $\mathfrak p\mid\lambda\mathcal O_K$. Together with the Archimedean places, this shows that $\Sigma_L$ is precisely the set of places of $L$ lying above those in $\Sigma_K$.

    For $v\in\Sigma_K$, let $t_v$ be the number of places of $L$ above $v$. Applying the overring and finite-unit-index arguments over $K$, with $\eta$ in place of $\alpha$, and then the $S$-unit theorem gives
    \[
        \rank S_\eta^\times = |\Sigma_K| - 1.
    \]
    Here, the denominator primes of $\eta=\lambda^{-1}$ are exactly the prime divisors of $\lambda\mathcal O_K$, and the equality $S_\eta = \zz[\eta] \cap \rr_{\ge0}$ gives $S_\eta^\times = \zz[\eta]^\times \cap \rr_{>0}$. The index calculation above and the injectivity of the logarithmic map on positive units, therefore, yield
    \[
        \dim W = \rank S_\alpha^\times=|\Sigma_K|-1.
    \]
    Let $v_0$ be the real place of $K$ represented by the identity embedding. For an embedding $\sigma:L\hookrightarrow\cc$ with restriction $\tau:=\sigma|_K$, we have
    \[
        |\sigma(\alpha)|^g=|\tau(\lambda)|^{-1}.
    \]
    If $\tau$ is the identity, then $|\sigma(\alpha)|=\alpha$; otherwise, the Perron inequality gives $|\sigma(\alpha)|>\alpha$. Thus, $\pi$ retains exactly the places above $v_0$. For $u=a\alpha^i\in S_\alpha^\times$, with $a\in S_\eta^\times$, each embedding above $v_0$ fixes $a$ and sends $\alpha$ to a number of modulus $\alpha$. Consequently, $\pi(\log u)$ is a scalar multiple of the vector with entry~$1$ at real places and entry~$2$ at complex places. Since $\alpha\in S_\alpha^\times$ and $\alpha\ne1$, its logarithmic image gives a nonzero such vector, and hence $\dim\pi(W)=1$.

    At least one coordinate is omitted by $\pi$. Indeed, if $K\ne\qq$, then $K$ has a real embedding and degree greater than one, so it has an Archimedean place other than $v_0$. If $K=\qq$, then the Perron number $\lambda$ is an integer greater than~$1$, and therefore has a prime divisor. The omitted coordinates are indexed by the places above $v\in\Sigma_K\setminus\{v_0\}$ and satisfy one independent zero-sum relation. Thus
    \[
        \dim \ker \pi = \sum_{v \in \Sigma_K \setminus \{v_0\}} t_v-1,
        \qquad
        \dim(W\cap\ker\pi) = \dim W - \dim \pi(W)= |\Sigma_K| - 2.
    \]

    The natural isomorphism $(\ker\pi+W)/W\cong\ker\pi/(W\cap\ker\pi)$ now gives
    \[
        \dim \big((\ker \pi + W)/W \big)
        = \sum_{v \in \Sigma_K \setminus \{v_0\}}(t_v - 1).
    \]
    Every place of $K$ extends to $L$, so each $t_v\ge1$. The claim therefore shows that $S_\alpha$ is an FFS if and only if $t_v=1$ for every $v \in \Sigma_K \setminus \{v_0\}$.

    To interpret the finite-place condition, let $\mathcal P_K$ and $\mathcal P_L$ be the sets of distinct nonzero prime ideals dividing $\lambda\mathcal O_K$ and $\lambda\mathcal O_L$, respectively. The valuation calculation above shows that $\mathcal P_L$ is the disjoint union of the sets of primes above $\mathfrak p\in\mathcal P_K$. Hence
    \[
        |\mathcal P_L|-|\mathcal P_K|
        = \sum_{\mathfrak p\in\mathcal P_K}(t_{\mathfrak p}-1).
    \]
    Thus $t_{\mathfrak p}=1$ for every $\mathfrak p\in\mathcal P_K$ if and only if the two principal ideals have the same number of distinct prime divisors.

    It remains to count the Archimedean places above $v\ne v_0$. Choose an embedding $\tau:K\hookrightarrow\cc$ representing $v$. Since $x^g-\eta$ is the minimal polynomial of $\alpha$ over $K$, there are exactly $g$ embeddings $\sigma:L\hookrightarrow\cc$ extending $\tau$, corresponding to the distinct roots of $x^g-\tau(\eta)$. If $v$ is complex, complex conjugation pairs each embedding above $\tau$ with an embedding above $\bar\tau$. Each complex place above $v$ contains exactly one embedding above $\tau$, so $t_v=g$.

    If $v \ne v_0$ is real, then $\tau(\eta) < 0$. Otherwise the positive $g$th root of $\tau(\eta)$ would be a positive conjugate of $\alpha$ distinct from $\alpha$, contrary to the uniqueness already established. For even $g$, the polynomial $x^g-\tau(\eta)$ has no real roots and gives $g/2$ complex places. For odd $g$, it has one real root and $(g-1)/2$ nonreal conjugate pairs. Hence in either case
    \[
        t_v=\lceil g/2\rceil.
    \]

    For $g=1$, we have $L=K$, so every place count is one and the finite-place condition is automatic. For $g=2$, every other real place has one extension, whereas a complex place has two; thus the Archimedean condition is equivalent to $K$ being totally real. For $g\ge3$, every Archimedean place other than $v_0$, if present, has at least two extensions. The condition is therefore equivalent to $v_0$ being the only Archimedean place of $K$. Because $v_0$ is real, the signature formula then gives $[K:\qq]=1$, or $K=\qq$. Combining these alternatives with the finite-place condition proves~(2).
\end{proof}

We conclude this section with another characterization of the finite factorization property over the class of algebraic monogenic semidomains. This characterization was first established by the second author, Hong, and Li in a project that is still in progress.

\begin{prop}
    Let $\alpha$ be an algebraic number, and set
    \[
        U_\beta := \{ u \in \zz[\alpha]^\times : u\beta, u^{-1}\beta \in S_\alpha \}.
    \]
    for every $\beta \in S_\alpha^*$. Then $S_\alpha$ is an FFS if and only if $U_\beta/S_\alpha^\times$ is finite for every nonzero $\beta \in S_\alpha$.
\end{prop}

\begin{proof}
    For the direct implication, assume that $S_\alpha$ is an FFS. Fix $\beta \in S_\alpha^*$ and let us argue that $U_\beta/S_\alpha^\times$ is finite. Let $\mathcal{D}_{\beta^2}$ denote the $S_\alpha$-associate classes of the set of divisors of $\beta^2$ in $S_\alpha$. Define the map $U_\beta/S_\alpha^\times \to \mathcal{D}_{\beta^2}$ by $uS_\alpha^\times \mapsto u\beta$. To see that this is a valid definition, it suffices to observe that for each $u \in U_\beta$ the divisibility relation $u \beta \mid_{S_\alpha} \beta^2$ holds because $u\beta \in S_\alpha$ and $\beta^2/u\beta = u^{-1}\beta \in S_\alpha$. In addition, as $\beta \neq 0$ the defined map is injective. As $S_\alpha$ is an FFS, $\mathcal{D}_{beta^2}$ is finite, and the injectivity of the map $U_\beta/S_\alpha^\times \to \mathcal{D}_{\beta^2}$ ensures that the orbit set $U_\beta/S_\alpha^\times$ is also finite.

    For the other direction, assume that the quotient $U_\beta/S_\alpha^\times$ is a finite set for all $\beta \in S_\alpha^*$. Then fix $\gamma \in S_\alpha$. Proposition~\ref{prop:Z-alpha-is-FFD} gives finitely many $\zz[\alpha]$-associate classes of divisors of $\gamma$ in $\zz[\alpha]$: let them be $d_1, \dots, d_\ell \in S_\alpha$. Let $d$ be another such divisor$d = d_i u$, with $u \in \zz[\alpha]^\times$, write $b=dc$ with $c\in S_\alpha$. For $a_i=d_i b$, we have
    \[
        a_i u = bd \in S_\alpha,
        \qquad a_i u^{-1} = d_i^2c\in S_\alpha.
    \]
    Hence $u \in U_{a_i}$. Each $S_\alpha^times$-orbit of such units contributes at most one associate class of divisors. Hence every $b$ has finitely many divisors up to associates, which proves that $S_\alpha$ is an FFS.
\end{proof}

The following example shows that there exist infinitely many positive $\alpha < 1$ of degree $3$ so that $S_\alpha$ is an FFS but not an integral domain.

\begin{example}
	Fix $t \in \zz$ with $t \ge 5$, and consider the polynomial $m_t(x) := x^3 + tx - 1$. A brief look at its derivative shows that $m_t(x)$ has a unique real root: $m_t'(x) = 3x^2+t > 0$ for all $x \in \rr$. Let $\alpha$ be the unique real root of $m_t(x)$, and notice that $\alpha > 0$ because $m_t(0) = -1 < 0$. On the other hand, $0 < m_t(1/t) = 1/t^3$, from which we conclude that $0 < \alpha < 1/t \le 1/5$. The rational root theorem shows that $m_t(x)$ is irreducible over $\qq$, and so $[\qq(\alpha):\qq] = 3$. Its other two roots form a nonreal conjugate pair, which we denote by $\beta$ and $\overline{\beta}$.

	We first determine the units of $S_\alpha$. The identity $\alpha(\alpha^2+t)=1$ shows that $\alpha$ is a unit of $\zz[\alpha]$. Since the order $\zz[\alpha]$ has signature $(1,1)$, Dirichlet's Unit Theorem gives
	\[
		\zz[\alpha]^\times \cong \{\pm1\} \times \zz.
	\]
	Let $u>1$ generate the group of positive units of $\zz[\alpha]$. Since $0 < \alpha < 1$ is a positive unit, there exists $m \in \nn$ such that $\alpha = u^{-m}$. Suppose, by way of contradiction, that $m \ge 2$. Then $u^m = \alpha^{-1} = t+\alpha^2$. As $t < \alpha^{-1} < t+1$ and $m \ge 2$, we obtain
	\[
		1 < u \le \sqrt{\alpha^{-1}} < \sqrt{t+1} < t.
	\]

	Write $\beta=\rho e^{i\theta}$, where $\rho>0$ and $\theta \in (0, \pi)$. Since $N(\alpha) = 1$, we see that $\rho=|\beta|=\alpha^{-1/2}$. Vieta's formulas give $\alpha + \beta + \overline{\beta} = 0$, and therefore
	\[
		\rho\cos\theta = -\frac{\alpha}{2}
		\qquad\text{and}\qquad
		\rho\sin\theta = \sqrt{\alpha^{-1}-\frac{\alpha^2}{4}}.
	\]
	Write $u = a+b \alpha + c\alpha^2$ for some $a,b,c \in \zz$, and set $u_\beta := a+b\beta + c\beta^2$. As $u$ is a unit, $|N(u)|=1$, and hence $1=|N(u)|=u|u_\beta|^2$. Thus, $|u_\beta| = u^{-1/2} < 1$. Direct calculation gives
	\[
		\Re(u_\beta) = a-\frac{\alpha}{2}b+\left(\frac{\alpha^2}{2}-\alpha^{-1}\right)c
		\qquad\text{and}\qquad
		\Im(u_\beta) = (b-c\alpha)\rho\sin\theta.
	\]
	It follows that
	\[
		|b-c\alpha|
		\le \frac{|u_\beta|}{\rho\sin\theta}
		\le \frac{1}{\sqrt{t-1/(4t^2)}}
		< \frac12.
	\]
	Consequently, $|b|<|c|\alpha+\frac12$. In particular, if $|c|\le1$, then $|b|<1$, and so $b=0$.

	We now show that $|c|\ge2$ is impossible. Substituting $a=u-b\alpha-c\alpha^2$ into the expression for $\Re(u_\beta)$ yields
	\[
		\Re(u_\beta)=u-\frac{3\alpha}{2}b - \left(\alpha^{-1}+\frac{\alpha^2}{2}\right)c.
	\]
	Using $u<t$, $\alpha^{-1}=t+\alpha^2$, and $|b|<|c|\alpha+1/2$, we obtain
	\begin{align*}
		|\Re(u_\beta)| \ge |c|\left(\alpha^{-1} + \frac{\alpha^2}{2}\right) - u - \frac{3\alpha}{2}|b| > (|c|-1)t-\frac{3\alpha}{4} > 1,
	\end{align*}
	contrary to $|\Re(u_\beta)|\le|u_\beta|<1$. Hence $|c|\le1$, and the preceding paragraph gives $b=0$. Therefore
	\[
		u = a, \qquad u = a+\alpha^2, \qquad\text{or}\qquad u=a-\alpha^2
	\]
	for some $a\in\zz$.

	If $u=a$ then $1=|N(u)|=|a|^3$, contradicting $u>1$. If $u=a+\alpha^2$, then $1\le a\le t-1$, and so
	\[
		|u_\beta|=|a+\beta^2|
		\ge\big||\beta|^2-a\big|
		=\alpha^{-1}-a>1,
	\]
	again a contradiction. Finally, if $u = a-\alpha^2$, then $a \ge 2$, and
	\[
		|u_\beta| \ge |\Re(u_\beta)|
		= a + \alpha^{-1} - \frac{\alpha^2}{2} > 1,
	\]
	which is also impossible. Therefore $m=1$, and the positive units of $\zz[\alpha]$ are precisely the powers of~$\alpha$.

	Since $\alpha^{-1}=t+\alpha^2\in S_\alpha$, every integral power of $\alpha$ belongs to $S_\alpha$. Conversely, every unit of $S_\alpha$ is a positive unit of $\zz[\alpha]$. Thus,
	\[
		S_\alpha^\times=\{\alpha^j : j \in \zz \} \qquad\text{and} \qquad [\zz[\alpha]^\times : S_\alpha^\times] = 2.
	\]
	By Proposition~\ref{prop:Z-alpha-is-FFD}, the domain $\zz[\alpha]$ is an FFD. Hence every nonzero element of $S_\alpha$ has only finitely many divisor classes in $\zz[\alpha]$, and each such class splits into at most two associate classes in $S_\alpha$. The divisor-finiteness characterization~\cite{fHK92} now shows that $S_\alpha$ is an FFS. Since $S_\alpha\subseteq\rr_{\ge0}$, we have $-1\notin S_\alpha$, and so $S_\alpha$ is not an integral domain. Finally, the identity $t=\alpha^{-1}-\alpha^2$ shows that distinct values of $t$ yield distinct parameters~$\alpha$.
	\hfill $\blacksquare$
\end{example}

\medskip
\subsection{The Unique Factorization Propoerty}

In this last section we will take a look at the unique factorization property.

\begin{prop}\label{prop:rational elements of Z beta}
    Let $\beta$ be any algebraic number, let $n\in\nn$, and let $M=\omega_\beta\in\zz[x]$. Let $\delta$ be the positive generator of $(n\zz[x]+M\zz[x])\cap\zz$. Then, for every $k\in\zz$,
    \[
        \frac{k}{n}\in\zz[\beta]\quad\Longleftrightarrow\quad\delta\mid k.
    \]
    In particular, $1/n\in\zz[\beta]$ if and only if $n\zz[x]+M\zz[x]=\zz[x]$. If $\beta$ has no nonnegative real conjugate, $\zz[\beta]$ may be replaced throughout by $S_\beta$.
\end{prop}

\begin{proof}
    Since $n$ belongs to the intersection ideal, its positive generator exists and divides $n$. Gauss's lemma identifies the kernel of evaluation at $\beta$ with $M\zz[x]$. Therefore
    \begin{align*}
        k/n\in\zz[\beta]
        &\Longleftrightarrow nf(\beta)=k\text{ for some }f\in\zz[x]\\
        &\Longleftrightarrow k\in(n\zz[x]+M\zz[x])\cap\zz\\
        &\Longleftrightarrow\delta\mid k.
    \end{align*}
    The assertion for $1/n$ follows by setting $k=1$, and the semiring assertion follows from Proposition~\ref{prop:monogenic semidomains that are ID}.
\end{proof}

\begin{prop}\label{prop:localization preserves UFD}
    Let $\alpha\in\aaa$, $n\in\nn$, $\beta=\alpha/n$, and $M=\omega_\beta$. If
    \[
        n\zz[x]+M\zz[x]=\zz[x],
    \]
    then $\zz[\beta]=\zz[\alpha][1/n]$. Consequently, if $\zz[\alpha]$ is a UFD, so is $\zz[\beta]$. If $\alpha$ has no nonnegative real conjugate, the same assertions hold with $S_\alpha$ and $S_\beta$ in place of these rings.
\end{prop}

\begin{proof}
    Proposition~\ref{prop:rational elements of Z beta} gives $1/n\in\zz[\beta]$. Also $\alpha=n\beta\in\zz[\beta]$, so $\zz[\alpha][1/n]\subseteq\zz[\beta]$. The reverse inclusion follows from $\beta=\alpha/n\in\zz[\alpha][1/n]$. Localizations of UFDs are UFDs. The final assertion follows because positive scaling preserves the existence of a nonnegative real conjugate.
\end{proof}

\begin{prop}\label{prop:unconditional-UFD-ascent}
    Let $\alpha$ be an algebraic integer. If $\zz[\alpha]$ is a UFD, then $\zz[\alpha/n]$ is a UFD for every $n\in\nn$. Consequently, if $S_\alpha$ is a UFD, then $S_{\alpha/n}$ is a UFD for every $n\in\nn$.
\end{prop}

\begin{proof}
    Set $R=\zz[\alpha]$ and $K=\qq(\alpha)$. Since $R$ is a UFD, it is integrally closed in $K$. Every element of $\mathcal O_K$ is integral over $R$, while $R\subseteq\mathcal O_K$; hence $R=\mathcal O_K$. A Dedekind UFD is a PID: its nonzero prime ideals are generated by prime elements, and every nonzero ideal is a product of such ideals.

    Write $\beta=\alpha/n=a/b$ with $a,b\in R$, $b\ne0$, and $a,b$ relatively prime in this PID. There exist $x,y\in R$ with $ax+by=1$. As $R\subseteq\zz[\beta]$ because $\alpha=n\beta$, we obtain
    \[
        b^{-1}=\beta x+y\in\zz[\beta].
    \]
    Therefore $R[1/b]\subseteq\zz[\beta]$. Conversely, $\beta=a/b\in R[1/b]$, so $\zz[\beta]=R[1/b]$, which is a UFD. If $S_\alpha$ is a UFD, it equals $R$, and Proposition~\ref{prop:monogenic semidomains that are ID} gives $S_\beta=\zz[\beta]$ as well.
\end{proof}

\begin{example}\label{example: i/2 unit group}
    We proceed to argue that the group of units of the monoid $\nn_0[i/2]$ can be generated by the set $\{i+1,2\}$, which has minimal size.

    We will first find all the elements of the unit group $\nn_0[i/2]^{\times}$. The primitive minimal polynomial of $i/2$ is $M(x) = 4x^2+1$, and
    \[
        1 = M(x)-2(2x^2).
    \]
    Therefore Proposition~\ref{prop:rational elements of Z beta} yields $1/2\in \nn_0[i/2]$, meaning $2^k \in \nn_0[i/2]^{\times}$ for all $k\in\zz$. Note that every element of $\nn_0[i/2]$ can be written in the form $(a+bi)/2^k$ for $a,b,k\in\zz$. Without loss of generality, it suffices to consider only the case where $k=0$ and $2$ does not divide at least one of $a$ or $b$. Then, if we find that some $u\in\nn_0[i/2]$ of this form is a unit, then $u\cdot 2^n$ will be units for all $n\in \zz$, which includes our original element $(a+bi)/2^k$. We will form two cases: either one of $a$ or $b$ is $0$, or both $a,b\neq 0$.
	\smallskip

	\textsc{Case 1:} $a$ or $b$ is $0$. We only need to consider the case where $a$ or $b$ is an odd integer. In either case, the reciprocal is an element of $\nn_0[i/2]$ only if $a$ or $b$ is $\pm 1$, which means that $\pm 2^n$ and $\pm 2^ni$ are the only units given by this case.
	\smallskip

    \textsc{Case 2:} $a,b\neq 0$. We can assume that $2\nmid \gcd(a,b)$. Now if $u=a+bi$ is a unit, then $1/u=(a-bi)/(a^2+b^2)\in\nn_0[i/2]$. Suppose for the sa                                                ke of contradiction that $\gcd(a,b)\neq 1$, so some odd prime $p$ divides $\gcd(a,b)$. Then, once the fraction $(a-bi)/(a^2+b^2)$ is simplified, the denominator would contain a power of $p$ that is $\nu_p(a^2+b^2)-\min(\nu_p(a),\nu_p(b))\geq 2\min(\nu_p(a),\nu_p(b)) - \min(\nu_p(a),\nu_p(b))$, which takes on a positive value when $p\mid \gcd(a,b)$. Since having an odd prime in the denominator would imply $1/u\not\in\nn_0[i/2]$, we must therefore have $\gcd(a,b)=1$. From here, we see that $a^2+b^2=2^n$ for some $n\geq 1$. Since we know that $a$ and $b$ are not both even, they both must be odd due to the parity of $2^n$. Then $a^2+b^2\equiv2\pmod 4$, which forces $n=1$. Thus, the only possible units this case gives are $\pm 2^n(i+1)$ and $\pm2^n(i-1)$ for $n\in\zz$, and it is easy to check that these are indeed units.

    Now, it is not hard to see that $i+1$ and $2$ generate $\nn_0[i/2]^{\times}$. We can first get all $2^n$, $n\in\zz$ using $2$. Then, we have $2^{-1}\cdot (i+1)^2=i$, and $i\cdot (i+1)=i-1$. Therefore, we can form all  $\pm 2^n(i+1)$, $\pm2^n(i-1)$, $\pm2^n$, $\pm2^ni$.

    To show that two generators are necessary, note that
    \[
        \nn_0[i/2]=\zz[i,1/2],\qquad
        \nn_0[i/2]^\times=\langle i\rangle\times\langle1+i\rangle
        \cong C_4\times\zz.
    \]
    Indeed, $2=-i(1+i)^2$, and the preceding list of units gives the claimed product. Its factors intersect trivially because $|1+i|>1$. This infinite group has a nontrivial torsion element $i$, so it cannot be cyclic. Hence the displayed two-element generating set $\{1+i,2\}$ has minimal size.

    Furthermore, we can note that $i+1$ is an atom of $\nn_0[i]$ but a unit in $\nn_0[i/2]$.
	\hfill $\blacksquare$
\end{example}

We conclude with an explicit nonrational algebraic monogenic semidomain that is an FFS and admits factorizations into atoms of different lengths.

\begin{example}
	Let $\alpha:=1/\sqrt{2}$. As observed in Example~\ref{ex:units of S alpha sqrt2}, the semidomain $S_\alpha$ is not a ring and
	\[
		S_\alpha^\times=\langle\alpha\rangle=\langle\sqrt{2}\rangle\cong\zz.
	\]
	We proceed to show that $S_\alpha$ is an FFS that is not half-factorial. For this, let $R$ denote the ring $\zz[1/2,\sqrt{2}]$, and then consider the following elements of $R$:
	\[
		\varepsilon:=1+\sqrt{2},\qquad
		\pi:=3+\sqrt{2},\qquad \text { and } \qquad
		\overline{\pi}:=3-\sqrt{2}.
	\]
	The ring $R$ is a localization of the Euclidean domain $\zz[\sqrt{2}]$ at the multiplicative set of powers of~$2$, whence $R$ is a UFD. Moreover,
	\[
		R^\times=\langle-1\rangle\times\langle\sqrt{2}\rangle\times\langle\varepsilon\rangle,
	\]
	while $\pi$ and $\overline{\pi}$ are non-associate primes of~$R$ satisfying $\pi\overline{\pi}=7$.

	Let $\sigma$ be the nontrivial automorphism of $\qq(\sqrt{2})$. For a positive element $x=a+b\sqrt{2}\in R$, one has $x\in S_\alpha$ if and only if $a,b\ge0$, or equivalently, if and only if $x\ge|\sigma(x)|$.

	We first prove that $S_\alpha$ is an FFS. Fix a nonzero $x\in S_\alpha$. Since $R$ is a UFD, the divisors of $x$ in~$R$ lie in finitely many associate classes. Fix one such class containing a divisor $d\in S_\alpha$ of~$x$, and write $x=dc$ with $c\in S_\alpha$. Every divisor of $x$ in $S_\alpha$ that is associate to $d$ in~$R$ has the form
	\[
		(\sqrt{2})^n\varepsilon^m d
	\]
	for some $m,n\in\zz$. Since $(\sqrt{2})^n\in S_\alpha^\times$, its associate class in $S_\alpha$ has a representative of the form $\varepsilon^m d$. For every positive $y\in R$, set
	\[
		\lambda(y):=\frac{\log\big(y/|\sigma(y)|\big)}{2\log\varepsilon}.
	\]
	The membership criterion above gives
	\[
		\varepsilon^m d\in S_\alpha
		\quad\Longleftrightarrow\quad
		m\ge-\lambda(d),
	\]
	while the corresponding cofactor belongs to $S_\alpha$ precisely when
	\[
		\varepsilon^{-m}c\in S_\alpha
		\quad\Longleftrightarrow\quad
		m\le\lambda(c).
	\]
	Thus, only finitely many integers $m$ can occur. Therefore each associate class of divisors of~$x$ in~$R$ contains only finitely many associate classes of divisors in~$S_\alpha$. It follows that $x$ has only finitely many divisors up to associates in $S_\alpha$. The divisor-finiteness characterization of finite factorization monoids~\cite{fHK92} now shows that $S_\alpha$ is an FFS.

	We now show that $S_\alpha$ is not half-factorial. After disregarding powers of the unit~$\sqrt{2}$, the element $\varepsilon^m\pi^i\overline{\pi}^{\,j}$ belongs to $S_\alpha$ precisely when
	\[
		m+\theta(i-j)\ge 0, \qquad \text{where}
		\qquad
		\theta:=\frac{\log(\pi/\overline{\pi})}{2\log\varepsilon}.
	\]
	The inequalities $\pi-\varepsilon\overline{\pi}=2-\sqrt{2} > 0$ and $\varepsilon^4\overline{\pi}^{\,3}-\pi^3=24+18\sqrt{2} > 0$ show that $1/2 < \theta \in (1/2, 2/3)$. Now consider the following elements:
	\begin{align*}
		u_1&:=10+\sqrt{2}=\sqrt{2}\,\varepsilon^{-1}\pi^2,
		&u_2&:=6+19\sqrt{2}=\sqrt{2}\,\varepsilon^2\overline{\pi}^{\,3},\\
		u_3&:=7\sqrt{2}=\sqrt{2}\,\pi\overline{\pi},
		&u_4&:=1+2\sqrt{2}=\varepsilon\overline{\pi}.
	\end{align*}
	The exponent triples $(m,i,j)$  of these four elements are $(-1,2,0)$,$(2,0,3)$, $(0,1,1)$, and $(1,0,1)$, respectively. Here a factor has $i,j\in\nn_0$ and $m\in\zz$. Its least allowed $m$ is $\lceil-\theta(i-j)\rceil$; the zero triple corresponds to an $S_\alpha$-unit. For $u_1$, a split of its two $\pi$ factors requires $m\ge0+0$, exceeding $-1$. For $u_2$, a split of its three $\overline\pi$ factors requires $m\ge1+2=3$, exceeding $2$. For $u_3$, separating $\pi$ and $\overline\pi$ requires $m\ge0+1$, exceeding $0$. The element $u_4$ has only one prime factor in $R$. In each of the four cases the displayed $m$ is already minimal, so a factor with $i=j=0$ and $m\ge1$ is also impossible. Since $R$ is a UFD, these exhaust the possible decompositions into two nonunits in $S_\alpha$. Thus all four elements are atoms. 
	Finally,
	\[
		u_1u_2=(10+\sqrt{2})(6+19\sqrt{2})
		=98+196\sqrt{2}
		=(7\sqrt{2})^2(1+2\sqrt{2})=u_3^2u_4.
	\]
	Hence the same element has a factorization of length~$2$ and a factorization of length~$3$. Consequently, $S_\alpha$ is an FFS that is not half-factorial and so no factorial.
\end{example}





\bigskip
\section*{Acknowledgments}

This paper is the result of a collaboration carried out while the authors where part of the first CrowdMath Internship, and the authors would like to thank the CrowdMath organizers and directors for making this research opportunity possible. While working on this paper, the first author was kindly supported by the NSF awards DMS-2213323.

\bigskip

\end{document}